\documentclass{amsart}
\usepackage{graphicx} 
\usepackage{amsmath,dsfont}
\usepackage{amsthm}
\usepackage{amssymb}
\usepackage{xcolor} 
\usepackage{mathtools}
\usepackage{colortbl} 
\usepackage{tikz} 
\usepackage{hf-tikz} 
\usepackage{makecell}
\usepackage{mathrsfs}
\usepackage{graphicx}
\usepackage{enumerate}\usepackage{chngcntr}
\usepackage{tikz}
\usepackage{enumitem}
\usepackage[colorlinks,
linkcolor=red,
anchorcolor=blue,
citecolor=green
]{hyperref}
\newtheorem{theorem}{Theorem}
\newtheorem*{theorem*}{Theorem}
\numberwithin{theorem}{section}
\newtheorem{definition}{Definition}[section]
\newtheorem{prop}{Proposition}[section]
\newtheorem{lemma}{Lemma}[section]
\newtheorem{remark}{Remark}[section]
\DeclareMathOperator{\supp}{supp}

\newtheorem{conjecture}{Conjecture}[section]

\newtheorem*{claim}{Claim}
\newcommand{\R}{\mathbb{R}}

\numberwithin{prop}{section}
\counterwithout{table}{section}
\DeclareMathOperator{\cell}{cell}

\title[Curved Kakeya sets for generic phases in odd dimensions]{Curved Kakeya sets for generic phases\\
	in odd dimensions}
\author{Shaoming Guo}
\address{Chern Institute of Mathematics and LPMC, Nankai University, Tianjin 300071, PR China}
\email{Shaomingguo2018@gmail.com}
\author{Diankun Liu}
\address{School of Mathematical Sciences, Zhejiang University, Hangzhou 310027, PR China}
\email{liudiankun5@gmail.com}
\author{Yakun Xi}
\address{School of Mathematical Sciences, Zhejiang University, Hangzhou 310027, PR China}
\email{yakunxi@zju.edu.cn}
\begin{document}
	
	\begin{abstract}
		We show that for each odd integer $n\ge 3$, there is an open dense subset of H\"ormander phase functions in $\mathbb{R}^n$ for which the associated curved Kakeya sets have Hausdorff dimension at least $\frac{n+1}{2} + d_n$ for some positive $d_n$, thereby exceeding the classical compression threshold. In particular, in $\mathbb{R}^3$, generic H\"ormander phases induce curved Kakeya sets of dimension at least $2 + \tfrac17$. As an application, on a generic three-dimensional Riemannian manifold, a local Nikodym set has Hausdorff dimension at least $2 + \tfrac17$. We achieve these results by generalizing the finite contact order condition from \cite{DGGZ24}, originally developed in $\mathbb{R}^3$, to arbitrary dimensions. Our bounds are stronger than those of \cite{DGGZ24} even in $\mathbb{R}^3$, since we derive curved Kakeya estimates directly via the polynomial method. Moreover, for H\"ormander-type oscillatory integral operators with positive-definite phases of finite contact order, we obtain quantitative improvements in all odd dimensions over the bounds of \cite{GHI}, while in three dimensions our oscillatory integral estimate exactly matches the result of \cite{DGGZ24}.
	\end{abstract}
	
	\maketitle
	
	\section{Introduction}
	\subsection{Curved Kakeya sets}
	
	The curved Kakeya problem asks for lower bounds on the Hausdorff dimension of subsets of $\mathbb{R}^n$ that contain, in every ``direction'', a curve drawn from a family determined by a single phase function satisfying H\"ormander’s non-degeneracy condition.
	
	\begin{definition}
		[H\"{o}rmander phase function] Let $\phi\in C^\infty( B_{\varepsilon_0}^{n-1}\times B_{\varepsilon_0}^{1}\times B_{\varepsilon_0}^{n-1})$, where $0<\varepsilon_0\ll 1$ is fixed throughout the paper. We say that $\phi$ is a H\"ormander phase function in $\mathbb R^n$ if for every
		$(\mathbf{x},y)\in B_{\varepsilon_0}^{n-1}\times B_{\varepsilon_0}^{1}\times B_{\varepsilon_0}^{n-1}$ the following hold:
		\\$(H_{1})$ $${\rm rank\,} \nabla_{\mathbf{x}}\nabla_{y}\phi(\mathbf{x} ;y)=n-1.$$ 
		\\$(H_{2})$  With 
		$$ G_{0}(\mathbf{x};y):=\bigwedge_{j=1}^{n-1}\partial_{y_{j}}\nabla_{\mathbf{x}}\phi(\mathbf{x};y),$$
		one has 
		$$\det\bigl[\nabla_{y'}^{2}\langle \nabla_{\mathbf{x}}\phi(\mathbf{x};y'),G_{0}(\mathbf{x};y)\rangle\bigr]_{y'=y}\neq 0.$$
		We write $\mathbf{H}$ for the space of all such H\"ormander phase functions.
	\end{definition}
	
	By a standard argument in \cite{Hor73} and \cite{Bourgain91}, a H\"ormander phase can be reduced to the normal form
	\[
	\phi(\mathbf{x};y)
	= \langle x,y\rangle + t\,\langle y,Ay\rangle
	+ O\bigl(|\mathbf{x}|^2\,|y|^2 + |t|\,|y|^3\bigr),
	\]
	where $(\mathbf{x};y)=(x,t;y)\in B_{\varepsilon_{\phi}}^{n-1}\times B_{\varepsilon_{\phi}}^{1}\times B_{\varepsilon_{\phi}}^{n-1}$ for some $\varepsilon_\phi\in(0,\varepsilon_0]$, and $A$ is a non-degenerate matrix.

	Every H\"ormander phase function $\phi$ naturally induces a family of Kakeya‐type curves \cite{Bourgain91,Wisewell05}.  
	
	\begin{definition}[Kakeya Curve]\label{curved}
		Given a Hörmander phase $\phi$, and a constant $\varepsilon_\phi\in(0,\varepsilon_0]$. For $y\in B_{\varepsilon_\phi}^{n-1}$, $\mathbf{x}\in B_{\varepsilon_\phi}^{n-1}\times B_{\varepsilon_\phi}^{1}$, and $0<\delta<\varepsilon_\phi$, define the Kakeya curves associated with $(\phi,\varepsilon_\phi)$ by
		\[
		\Gamma_y^\phi(\mathbf{x})
		= \bigl\{\mathbf{x}'\in B_{\varepsilon_\phi}^{n-1}\times B_{\varepsilon_\phi}^{1}
		: \nabla_{y}\phi(\mathbf{x}';y)
		= \nabla_{y}\phi(\mathbf{x};y)\bigr\},
		\]
		and the associated $\delta$-tubes by
		\[
		T_y^{\delta,\phi}(\mathbf{x})
		= \bigl\{\mathbf{x}'\in B_{\varepsilon_\phi}^{n-1}\times B_{\varepsilon_\phi}^{1}
		: |\nabla_{y}\phi(\mathbf{x}';y)
		- \nabla_{y}\phi(\mathbf{x};y)| < \delta\bigr\}.
		\]
	\end{definition}
	Here, $y$ is called the “direction” of the curve $\Gamma_y^\phi(\mathbf x)$ and $\mathbf x$ its “position”; by selecting one Kakeya curve for each direction, we obtain a curved Kakeya set.
	
	\begin{definition}[Curved Kakeya set]\label{curvedK}
		Let $\phi$ be as above. A set $E\subset\mathbb R^n$ is a \emph{curved Kakeya set} associated with $(\phi,\varepsilon_\phi)$ if for each $y\in B_{\varepsilon_\phi}^{n-1}$ there exists a basepoint $\omega\in B_{\varepsilon_\phi}^{n-1}$ such that \[\Gamma_y^\phi((\omega,0))\subset E.\]
	\end{definition}

	\subsection{The Hausdorff dimension of generic curved Kakeya sets}
	
	Curved Kakeya sets extend the definition of classical straight‐line Kakeya sets, but their fractal dimensions typically differ from those of the standard straight‐line Kakeya sets.
	
	For general curved Kakeya sets, Wisewell \cite{Wisewell05} proved a Hausdorff‐dimension lower bound of $\tfrac{n+1}{2}$ in odd dimensions $n$, while Bourgain--Guth \cite{BG11} raised the lower bound to $\tfrac{n+2}{2}$ for even $n$. These bounds are generally sharp: there are explicit phase functions whose Kakeya sets concentrate on a $\lceil\tfrac{n+1}{2}\rceil$-dimensional submanifold---this concentration constitutes the Kakeya compression phenomenon. See \cite{Bourgain91,MS,BG11,Sogge-Xi,GHI}.
	
	Our main result asserts that, under the usual Fr\'echet topology of $C^\infty$ functions, there exists an open dense subset $\mathbf C \subset \mathbf H$ such that for every $\phi \in \mathbf C$, the associated curved Kakeya sets in $\mathbb R^n$ strictly exceed the Kakeya compression threshold $\tfrac{n+1}{2}$ for any odd $n$.
	
	\begin{theorem}\label{thm generic}
		For each odd integer $n\ge3$, there exists an open dense subset $\mathbf{C}\subset\mathbf{H}$ and an absolute constant $0<d_n<\frac12$ such that, for every $\phi\in\mathbf{C}$, there exists a constant $\varepsilon_{\phi}>0$ for which the curved Kakeya set $K\subset\mathbb{R}^n$ associated with $(\phi,\varepsilon_\phi)$ satisfies
		\[
		\dim_{\mathcal{H}}K \ge \frac{n+1}{2} + d_n.
		\]
		In particular, curved Kakeya sets in $\mathbb{R}^3$ induced by generic H\"ormander phases have Hausdorff dimension $\ge 2+\tfrac{1}{7}$.
		
	\end{theorem}
	
	It is curious to note that, in $\mathbb{R}^3$, the lower bound derived from Theorem \ref{thm generic} coincides with the lower bound of $\frac{4n+3}{7}=\frac{15}7$ established by Katz and Tao \cite{Katz-Tao} for standard Kakeya sets.
	
	We prove Theorem \ref{thm generic} by generalizing the notion of \emph{contact order} for H{\"o}rmander phase functions from $\mathbb{R}^3$ to arbitrary dimension $n\ge3$; see Section \ref{contact order} for detailed definitions. In the proof, we choose $\mathbf{C}$ to be the family of phase functions with the minimal required contact order $A(n)$. For the exact value of $A(n)$, see \eqref{eq An} in the appendix. The concept of contact order in $\mathbb{R}^3$ was introduced by Bourgain \cite{Bourgain91} and further developed in \cite{Sogge99,DGGZ24,CGGHIW24}.

	Under the finite contact order assumption, we show that the associated tube families satisfy certain non-trivial polynomial Wolff axioms, which counteract the Kakeya compression phenomenon. Moreover, given an upper bound $l$ on the contact order, we obtain quantitative improvements to the Hausdorff‐dimension lower bounds for corresponding curved Kakeya sets.
	
	\begin{theorem}\label{thm Kakeya}
		Let $n\ge3$ be odd, and let $\phi$ satisfy $(H_1)$ and $(H_2)$ and have contact order at most $l$ at the origin in $\mathbb{R}^n$ for some integer $l\geq A(n)$. Then there exists a constant $\varepsilon_{\phi}>0$ such that
		the curved Kakeya sets associated with $(\phi,\varepsilon_\phi)$ satisfy
		\[
		\dim_{\mathcal{H}}K \ge \frac{n+1}{2} + \frac{n-1}{2(l+2-n)(n-1)+2},
		\]
		where $\dim_{\mathcal{H}}K$ denotes the Hausdorff dimension of the curved Kakeya set $K$.
	\end{theorem}
	
	It is noteworthy that as $l\to\infty$, the improvement vanishes and the Kakeya compression phenomenon may occur. In the proof of Theorem \ref{thm generic}, we take $l = A(n)$ and
	\[
	d_n \coloneqq \frac{n-1}{2(A(n)+2-n)(n-1)+2},
	\]
	where $A(n)$ is the minimal required contact order. In $\mathbb{R}^3$, Dai, Gong, Guo, and Zhang \cite{DGGZ24} show that H{\"o}rmander phases with contact order $4=A(3)$ induce curved Kakeya sets of Hausdorff dimension at least $2+\tfrac{1}{13}$. Rather than studying the associated oscillatory integral operator and deducing Kakeya estimates as in \cite{DGGZ24}, we employ the polynomial method to derive curved maximal Kakeya estimates directly, obtaining the stronger lower bound $2+\tfrac{1}{7}$. As an application of Theorem \ref{thm Kakeya}, we derive corresponding results for Nikodym sets on Riemannian manifolds; see Theorem \ref{thm Manifold}. We show that, on a generic three‑dimensional Riemannian manifold, every Nikodym set has Hausdorff dimension at least $2+\frac17$.

	\subsection{H\"ormander's oscillatory integral operators}
	
	The Hausdorff dimension of curved Kakeya sets is closely tied to $L^p$-bounds for H\"ormander's oscillatory integral operators. 
	\begin{definition}[H\"ormander oscillatory integral operator]
		Let $\phi$ be a H\"ormander phase function and fix $\varepsilon_\phi\in(0,\varepsilon_0]$. Let
		$a\in C_c^\infty\bigl(B^{n-1}_{\varepsilon_\phi}\times B^{1}_{\varepsilon_\phi}\times B^{n-1}_{\varepsilon_\phi}\bigr)$
		be the amplitude.
		For $\lambda>1$ and $(x,t)\in\R^{n-1}\times\R$, define
		\[
		T_\lambda^\phi f(x,t)
		=\int_{\R^{n-1}}e^{i\lambda\phi(x,t;y)}\,a(x,t;y)\,f(y)\,dy.
		\]
	\end{definition}
	
	In the study of H\"ormander oscillatory integral operators, one often imposes a positive‐definite condition on the H\"ormander phases.
	
	\begin{definition}[Positive definiteness condition]
		We say that the phase function $\phi$ satisfies condition $(H_2^+)$ if the matrix 
		\[\nabla_{y'}^{2}\langle\nabla_{\mathbf{x}}\phi(\mathbf{x};y'),\,G_{0}(\mathbf{x};y)\rangle\big|_{y'=y}\]
		is positive‐definite for every 
		$(\mathbf{x},y)\in B_{\varepsilon_0}^{n-1}\times B_{\varepsilon_0}^{1}\times B_{\varepsilon_0}^{n-1}$.
	\end{definition}
	
	It is known that the Kakeya compression phenomenon can still occur for positive‐definite H\"ormander phases. Specifically, there exist positive-definite phases for which the corresponding curved Kakeya sets concentrate on a $\lceil\frac{n+1}{2} \rceil$-submanifolds \cite{MS, BG11,GHI, Sogge-Xi}. We write $\mathbf H^+$ for the space of positive‐definite H\"ormander phase functions.
	
	A key problem for H\"ormander oscillatory integral operators is to determine the maximal range of $p$ such that, for every $\epsilon>0$, there is a constant $C_{\epsilon}>0$ so that
	\begin{equation}\label{eq restriction}
		\|T_{\lambda}^{\phi}f\|_{L^{p}( B_{\varepsilon_\phi}^{n-1}\times  B_{\varepsilon_\phi}^{1})}
		\le
		C_{\epsilon}\,\lambda^{-\frac{n}{p}+\epsilon}\,\|f\|_{L^{p}( B_{\varepsilon_\phi}^{n-1})}
		\quad\text{for all }\lambda\ge1.
	\end{equation}
	Let us recall the generally sharp $L^p$ estimates established in \cite{GHI} for H\"ormander oscillatory integral operators with positive-definite phases.
	
	\begin{theorem}[\cite{GHI}]\label{thm GHI}
		Suppose the phase function $\phi$ satisfies $(H_1)$ and $(H_2^+)$. Set $\varepsilon_\phi=\varepsilon_0$. Then for every $\epsilon>0$ there is a constant $C_{\epsilon}>0$ such that the estimate \eqref{eq restriction} holds for all $\lambda\ge1$, provided $p$ lies in the following range:
		\begin{itemize}
			\item If $n\ge3$ is odd, then 
			\[
			p\ge p(n):=2\frac{3n+1}{3n-3};
			\]
			\item If $n\ge4$ is even, then 
			\[
			p\ge p(n):=2\frac{3n+2}{3n-2}.
			\]
		\end{itemize}
		Moreover, these bounds on $p$ are in general sharp.
	\end{theorem}
	
	The Kakeya compression phenomenon ensures the sharpness of the exponent range $p$ in these estimates: there exist H{\"o}rmander phases $\phi\in\mathbf H^+$ for which \eqref{eq restriction} fails whenever $p<p(n)$. However, we show that \eqref{eq restriction} admits quantitative improvements in odd dimensions when $\phi$ is restricted to the generic subset $\mathbf C^+:=\mathbf C\cap \mathbf H^{+}\subset\mathbf H^+$.
	
	\begin{theorem}\label{thm Horgeneric}
		Let $n\ge 3$ be an odd integer. There exists an absolute constant $\zeta_n>0$ such that for every $\phi\in\mathbf C^+$ there exists a constant $\varepsilon_{\phi}\in(0,\varepsilon_0]$ with the following property: for any $\epsilon>0$ there is a constant $C_\epsilon>0$ such that \eqref{eq restriction} holds whenever
		\[
		p\ge 2\frac{3n+1}{3n-3}-\zeta_n.
		\]
	\end{theorem}

	We prove this by establishing oscillatory integral estimates for H{\"o}rmander phases of contact order $l$ in arbitrary odd dimensions. In the proof of Theorem \ref{thm Horgeneric}, we take $l=A(n)$ and 
	$$\zeta_{n}\coloneqq \frac{4}{(3n-3)^2\,A(n)+(3n-3)^2(2-n)+2(3n-3)},$$ where $A(n)$ is the minimal required contact order.
	
	\begin{theorem}\label{thm improvement}
		Let $n\ge 3$ be odd, and suppose $\phi$ satisfies $(H_1)$ and $(H_2^+)$ and has contact order at most $l$ at the origin in $\mathbb{R}^n$ for some integer $l\geq A(n)$. There exists a constant $\varepsilon_{\phi}\in(0,\varepsilon_0]$ with the following property: for every $\epsilon>0$ there is a constant $C_\epsilon>0$ such that the estimate \eqref{eq restriction} holds whenever
		\[
		p\ge p(n,l)
		:=2\frac{3n+1}{3n-3}
		-\frac{4}{(3n-3)^2\,l+(3n-3)^2(2-n)+2(3n-3)}.
		\]
	\end{theorem}

	When $n=3$, Theorem \ref{thm improvement} recovers the result of \cite{DGGZ24}. In arbitrary odd dimensions, it provides a quantitative refinement of the Guth--Hickman--Iliopoulou exponent $p(n)$. Moreover, the gain decreases as the contact order $l$ increases and vanishes as $l\to\infty$ as expected.

	\subsection{History and discussion}
	
	We revisit the history of the Kakeya problem and situate our results within the literature. We begin by recalling the standard Kakeya set in $\R^n$.
	
	\begin{definition}[Kakeya set]
		A compact set $E\subset\R^n$ is a Kakeya set if for every direction $e\in S^{n-1}$ there exists $x\in\R^n$ such that
		\[
		\{x + te : t \in [-\tfrac12, \tfrac12]\}\subset E.
		\]
	\end{definition}
	
	This configuration arises from the phase function 
	\begin{equation}\label{eq standard}
		\phi(x,t;y)=\langle x,y\rangle + t\,h(y),
	\end{equation}
	where $\nabla^2h(y)$ is nondegenerate. Since $h$ is independent of $(x,t)$, the associated Kakeya curves are straight lines. The famous Kakeya conjecture asserts that every standard Kakeya set in $\R^n$ has full Hausdorff dimension.
	
	\begin{conjecture}\label{Conj}
		Every Kakeya set in $\R^n$ has Hausdorff dimension $n$.
	\end{conjecture}
	
	Drury \cite{Drury} proved that every Kakeya set in $\R^n$ has Hausdorff dimension at least $\tfrac{n+1}{2}$ using the X-ray transform. Bourgain \cite{Bourgain91.2} showed that in $\R^3$ these sets have dimension at least $\tfrac{7}{3}$—the first result to exceed the Kakeya compression exponent—via a \emph{bush argument}. Wolff \cite{Wolff95} obtained the lower bound $\tfrac{n+2}{2}$ using a \emph{hairbrush argument}, matching the compression exponent in even dimensions and improving on it in odd dimensions. Katz and Tao \cite{Katz-Tao} later proved a lower bound of $(2-\sqrt{2})(n-4)+3$ for Kakeya sets in $\R^n$. In $\R^4$, Guth and Zahl \cite{Guth-Zahl} improved Wolff's bound to $3+\frac{1}{40}$. Since then, standard Kakeya sets have surpassed the Kakeya compression exponent in every dimension $n$. Moreover, significant progress has been made on the Kakeya Conjecture \ref{Conj} (see \cite{D,Cor77,Schlag,Bourgain99,KLT,LT,ZahlJoshua,Katz19,Zahl,katz21,HRZ}). Most notably, Wang and Zahl \cite{sticky,assouad,WangZahl} proved the Kakeya Conjecture in $\R^3$ through a series of groundbreaking works.

	For standard Kakeya sets, the Kakeya compression phenomenon does not occur, and these sets are expected to have full Hausdorff dimension in $\mathbb{R}^{n}$. However, in the curved case, the Kakeya compression phenomenon may appear. Bourgain was the first to discover the Kakeya compression phenomenon \cite{Bourgain91.2} from the analytical perspective. He constructed a H\"ormander phase function in $\mathbb{R}^3$ for which the inequality \eqref{eq restriction} 
	fails for all $p < 4$. This implies that the range $p \geq 4$ for the estimate \eqref{eq restriction} with general phases satisfying $(H_{1})$ and $(H_{2})$ is sharp in $\mathbb{R}^{3}$.
	Subsequently, Minicozzi and Sogge \cite{MS} constructed Riemannian manifolds $M^n$ whose geodesic Nikodym sets concentrate on a $\lceil\tfrac{n+1}{2}\rceil$-dimensional submanifold. These Nikodym sets coincide with the curved Kakeya sets when the phase is the Riemannian distance function, which is positive definite (see Section \ref{Application}). Bourgain and Guth \cite{BG11} and Guth, Hickman, and Iliopoulou \cite{GHI} gave explicit positive-definite examples, showing that for some $\phi\in\mathbf H^+$ the induced curved Kakeya set has Hausdorff dimension $\lceil\tfrac{n+1}{2}\rceil$ and that \eqref{eq restriction} fails whenever $p<p(n)$.

	Several studies have addressed the Kakeya compression phenomenon in the variable-coefficient setting. Wisewell \cite{Wisewell05} proved that certain hyperbolic Kakeya sets in $\R^n$ satisfy $\dim_{\mathcal H}E\ge\tfrac{n+2}{2}$. Sogge \cite{Sogge99} and Xi \cite{Xi17} adapted Wolff’s hairbrush argument to show that geodesic Nikodym sets on constant-curvature manifolds $M^n$ likewise satisfy $\dim_{\mathcal H}E\ge\tfrac{n+2}{2}$. More recently, the first author, Wang, and Zhang \cite{GWZ22} introduced a unified condition on H\"ormander phases—now called \emph{Bourgain’s condition}—which encompasses these cases and offers a closer generalization of standard Kakeya sets.

	\begin{definition}[Bourgain's condition \cite{Bourgain91,GWZ22}]
		Let $\phi$ be a H\"ormander phase function.  We say that $\phi$ satisfies Bourgain's condition at $(\mathbf x_0;y_0)$ if for all $1\le i,j\le n-1$,
		\[
		(G_0\!\cdot\!\nabla_{\mathbf x})^2\partial^2_{y_i y_j}\phi(\mathbf x_0;y_0)
		= C(\mathbf x_0;y_0)\,(G_0\!\cdot\!\nabla_{\mathbf x})\partial^2_{y_i y_j}\phi(\mathbf x_0;y_0),
		\]
		where $C(\mathbf x_0;y_0)$ is a scalar function of $(\mathbf x_0;y_0)$.
	\end{definition}
	
	The first author, Wang, and Zhang proved that under Bourgain's condition the (strong) polynomial Wolff axioms hold, providing an effective tool to counteract the Kakeya compression phenomenon. Moreover, when both Bourgain's condition and $(H_2^+)$ are imposed, they showed that for every $\epsilon>0$ there is a constant $C_\epsilon>0$ such that
	\[
	p>p_{\mathrm{GWZ}}(n):=2+\frac{2.5921}{n}+O(n^{-2})
	\]
	implies \eqref{eq restriction} for all $\lambda\ge1$. This yields an improved exponent range for the Fourier restriction conjecture in higher dimensions. They further conjectured that if $\phi$ satisfies Bourgain's condition at every point then \eqref{eq restriction} holds for all $p\ge2n/(n-1)$, which would imply Conjecture \ref{Conj} for the corresponding curved Kakeya sets.

	In \cite{DGGZ24}, Dai, Gong, the first author, and Zhang studied curved Kakeya sets for phase functions satisfying Bourgain’s condition. They showed that such sets enjoy the Hausdorff-dimension lower bound from \cite{HRZ}, which was originally proven only for standard Kakeya sets. They also proved that a real-analytic Riemannian distance function satisfies Bourgain’s condition if and only if the manifold has constant sectional curvature. Moreover, as mentioned earlier, they introduced the contact order condition for H\"ormander phases in $\R^3$. Under the finite contact order assumption, they obtained quantitative improvements over the generally sharp exponent range $p\ge\tfrac{10}{3}$ for positive-definite phases in $\mathbb{R}^3$, yielding a Hausdorff dimension lower bound of $2+\tfrac{1}{13}$ for the associated curved Kakeya sets.

	Recently, Chen, Gan, the first author, Hickman, Iliopoulou, and Wright \cite{CGGHIW24} investigated real-analytic, translation-invariant H\"ormander phase functions. They introduced a Kakeya non-compression condition—a relatively weak restriction that prevents curved Kakeya sets from concentrating on lower-dimensional submanifolds Under this condition, they showed that the Hausdorff dimension of the associated curved Kakeya sets in $\mathbb{R}^n$ strictly exceeds $\tfrac{n+1}{2}$, improving on the Kakeya compression exponent for odd $n$.
	
	In the translation-invariant class (which is closed in the space of H\"ormander phase functions in the $C^\infty$ topology), finite contact order implies Kakeya non-compression. However, there are several key distinctions: finite contact order does not require translation invariance; it yields quantitative improvements rather than merely qualitative bounds; and it is stable under perturbations, whereas the non-compression condition is not.

	Gao, the second author, and the third author \cite{GaoLX} investigated two special classes of H\"ormander phase functions:
	\begin{itemize}
		\item the Riemannian distance function on constant‐curvature manifolds;
		\item translation‐invariant phases satisfying Bourgain’s condition.
	\end{itemize}
	They proved that for any $\phi$ in either class, the induced curved Kakeya set is diffeomorphic to a standard Kakeya set, and hence has the same Hausdorff dimension. Moreover, in $\R^3$, for translation‐invariant phases under Bourgain’s condition, \eqref{eq restriction} holds for all $p\ge3.25$ by combining \cite{BMV} and \cite{Guo-Oh}. Furthermore, under the additional positive-definiteness assumption, the estimate \eqref{eq restriction} holds for all $p\ge\tfrac{22}{7}$ \cite{WangWu2024}.

	More recently, Nadjimzadah \cite{Arian} introduced a new condition for translation‐invariant H\"ormander phases in $\R^3$ related to the $t$‐derivatives of the Hessian, and obtained a lower bound of $2.348$ for the Hausdorff dimension of the corresponding curved Kakeya sets.
	
	To summarize, in this paper we generalize the finite contact order condition introduced in $\mathbb{R}^3$ by Dai, Gong, the first author, and Zhang \cite{DGGZ24} to arbitrary dimension $n$ and show that it is generic. Under this assumption, we obtain:
	\begin{itemize}
		\item a quantitative improvement beyond the generally sharp Hausdorff dimension lower bound for curved Kakeya sets in odd dimensions (as in \cite{Wisewell05}), yielding $2+\tfrac{1}{7}$ in dimension 3 and improving the bound in \cite{DGGZ24};
		\item a quantitative improvement beyond the generally sharp exponent range for H{\"o}rmander-type oscillatory integral operators with positive-definite phases in odd dimensions (as in \cite{GHI}), recovering the result of \cite{DGGZ24} in dimension 3.
	\end{itemize}
	Finally, we remark that, morally speaking, Bourgain’s condition corresponds to the ``contact order being $n-1$''. In $\mathbb{R}^3$, setting $l=2$ in our proof yields a Hausdorff dimension bound of at least $7/3$, matching exactly the curved Kakeya bound from \cite{DGGZ24} under Bourgain’s condition. This agreement highlights the naturalness and robustness of our framework.

	\subsection{Structure of the paper}
	In Section \ref{reduction}, we reduce the proof of Theorem \ref{thm improvement} to a broad estimate for H\"ormander's oscillatory integral operators. In Section \ref{partitioning}, we introduce preliminaries on polynomial partitioning. In Section \ref{contact order}, we define the contact order for H\"ormander phases in arbitrary dimensions and show that the tube families induced by phases with finite contact order satisfy the polynomial Wolff axioms. In Section \ref{proof}, we establish the broad estimate for positive-definite phases with finite contact order in $\mathbb{R}^{n}$. In Section \ref{KKYM}, we focus on curved Kakeya sets and prove Theorem \ref{thm Kakeya} by establishing curved maximal Kakeya estimates. In Section \ref{generic}, we prove that the contact order condition is generic. In Section \ref{Application}, we present applications to Carleson--Sj\"olin operators and Nikodym sets on manifolds. Finally, in the Appendix, we analyze and compute the minimal required contact order $A(n)$ in $\mathbb{R}^{n}$.

	\subsection*{Acknowledgment}  This project is supported by the National Key Research and Development Program of China No. 2022YFA1007200. S. G. is partly supported by the Nankai Zhide Foundation, and partly supported by NSFC Grant No. 12426204. D. L. and Y. X. are supported by Zhejiang Provincial Natural Science Foundation of China under
	Grant No. LR23A010002, and NSFC Grant No.
	12171424.
	
	\subsection*{Notation} 
	We use $ B^m_r(x)$ to denote an $m$-dimensional ball of radius $r$ centered at $x$; we may simply write $ B^m_r$ or $ B_r$ if $x$ and $m$ are clear from context.
	
	We use $a\lesssim b$ to express that there exists a constant $C$ such that $a\le C\,b$, where $C$ is independent of all relevant parameters. 
	
	We use $a\sim b$ to express that there exist a constant $C>0$ such that $C^{-1}\,b\le a\le C\,b$, where $C$ is independent of all relevant parameters.
	
	We write $a\lesssim_{\alpha,\beta}b$ to indicate that there exists a constant $C_{\alpha,\beta}$, depending only on $\alpha,\,\beta$, such that $a\le C_{\alpha,\beta}\,b$. 
	
	Given a large parameter $R\ge1$ which should be clear from context, we use $a\lessapprox R^{\mu}b$ to denote that for every $\epsilon>0$, there exists a constant $C_{\epsilon}$ (independent of $R$) such that $a\le C_{\epsilon}\,R^{\mu+\epsilon}\,b$ for all $R\ge1$.
	
	Throughout the paper, for simplicity, the implicit and explicit constants are allowed to depend on the phase function $\phi$, which is fixed in the argument.

	\section{Standard reductions for H\"ormander oscillatory integral}\label{reduction}

	To prove Theorem \ref{thm improvement}, we consider the H\"ormander oscillatory integral operator with positive-definite phase:
	\begin{equation}
		T_{\lambda}^{\phi}f(x,t)
		=\int_{\mathbb{R}^{n-1}}e^{i\lambda\phi(x,t;y)}\,a(x,t;y)\,f(y)\,dy,
	\end{equation}
	where $\phi$ satisfies $(H_1)$ and $(H_2^+)$ and $a(x,t;y)$ supports in ${B}_{\varepsilon_{\phi}}^{n-1}\times{B}_{\varepsilon_{\phi}}^{1}\times{B}_{\varepsilon_{\phi}}^{n-1}$. Here $\varepsilon_{\phi}$ is a small constant that may depend on $\phi$.
	
	We perform the rescaling
	\begin{equation}
		\phi^{\lambda}(\mathbf{x};y)
		=\lambda\,\phi\Bigl(\frac{\mathbf{x}}{\lambda};y\Bigr),
		\quad
		a^{\lambda}(\mathbf{x};y)
		=a\Bigl(\frac{\mathbf{x}}{\lambda};y\Bigr),
	\end{equation}
	with $\lambda\ge1$ and $\mathbf{x}=(x,t)\in\mathbb{R}^{n-1}\times\mathbb{R}$. Define the rescaled operator
	\begin{equation}
		\mathcal{T}^{\lambda}f(x,t)
		=\int_{\mathbb{R}^{n-1}}e^{i\phi^{\lambda}(\mathbf{x};y)}\,a^{\lambda}(\mathbf{x};y)\,f(y)\,dy.
	\end{equation}
	
	Finally, set
	\begin{equation}
		\zeta(n,l)
		=\frac{4}{(3n-3)^2\,l+(3n-3)^2(2-n)+2(3n-3)}.
	\end{equation}
	To prove Theorem \ref{thm improvement}, it suffices to establish estimates for these rescaled operators.

	\begin{theorem}\label{thm osc}
		Let $n \geq 3$ be odd, and suppose that the phase function $\phi$ satisfies $(H_1)$ and $(H_2^{+})$ and has contact order at most $l$ at the origin in $\mathbb{R}^{n}$. Then there exists a constant $\varepsilon_\phi\in(0,\varepsilon_0]$ such that
		\begin{equation}
			\|\mathcal{T}^{\lambda}f\|_{L^{p}(\mathbb{R}^{n})} \lessapprox\|f\|_{L^{p}(\mathbb{R}^{n-1})}
		\end{equation}
		holds for all $\lambda \geq 1$, where the exponent $p$ satisfies:
		
		$$p \geq p(n,l):=2\,\dfrac{3n+1}{3n-3} - \zeta(n,l),$$
		
	\end{theorem}
	We note that the critical exponent tends to the exponent in \cite{GHI} when the contact order $l$ tends to $\infty$ for odd dimension $n$:
	\[
	\lim_{l \to \infty} p(n, l) = 
	2 \cdot \frac{3n + 1}{3n - 3} 
	\]

	In \cite{BG11,Guth18}, Bourgain and Guth decomposed the $L^p$ norm into \textbf{broad} and \textbf{narrow} parts, with the narrow part exhibiting better behavior in the induction process. To prove Theorem \ref{thm osc}, we first reduce to broad estimates by the Bourgain--Guth method.
	
	Before introducing the broad norms, we clarify some notation. Let $K\ge1$. We divide ${B}_{\varepsilon_{\phi}}^{n-1}$ into balls $\tau$ of radius $K^{-1}$. Let $f_{\tau}=f\cdot\chi_{\tau}$. Moreover, recall the Gauss map in $(H_{2})$:
	\[
	G_{0}(\mathbf{x};y)\coloneqq \bigwedge_{j=1}^{n-1}\partial_{y_{j}}\nabla_{\mathbf{x}}\phi(\mathbf{x};y),
	\]
	and define the normalized Gauss map
	\[
	G(\mathbf{x};y)=\frac{G_{0}(\mathbf{x};y)}{|G_{0}(\mathbf{x};y)|}.
	\]
	We also introduce the rescaled Gauss map
	\[
	G^{\lambda}(\mathbf{x};y)=G\Bigl(\frac{\mathbf{x}}{\lambda};y\Bigr).
	\]
	
	For $\mathbf{x}\in B_{R}$, denote
	\[
	G^{\lambda}(\mathbf{x},\tau)\coloneqq\{\,G^{\lambda}(\mathbf{x};y):y\in\tau\},
	\]
	where $B_{R}$ denotes the ball of radius $R$ in $\mathbb{R}^{n}$ and $1\le R\le\lambda$. Let $V\subset\mathbb{R}^{n}$ be a linear subspace. We write
	$
	\angle\bigl(G^{\lambda}(\mathbf{x};\tau),V\bigr)
	$
	for the smallest angle between any nonzero vector in $G^{\lambda}(\mathbf{x};\tau)$ and any nonzero vector in $V$.
	We let $\mathcal{U}(\mathbf{x},K,V)$ be the collection of all $\tau$ satisfying
	\[
	\angle\bigl(G^{\lambda}(\mathbf{x};\tau),V\bigr)\ge K^{-1},
	\]
	and let $\mathcal{E}(\mathbf{x},K,V)$ be the collection of all $\tau$ satisfying
	\[
	\angle\bigl(G^{\lambda}(\mathbf{x};\tau),V\bigr)\le K^{-1}.
	\]
	
	We now introduce the broad norm. Let $A$ be a fixed integer to be chosen later. Fix a ball $B_{K^{2}}\subset B_{R}$ centered at $\mathbf{x}$, and define
	\[
	\mu_{\mathcal{T}^{\lambda}f}(B_{K^{2}})
	\coloneqq
	\min_{V_{1},\dots,V_{A}\in Gr(k-1,n)}
	\Bigl(\max_{\substack{\tau\in\mathcal{U}(\mathbf{x},K,V_{a})\\\text{for all }1\le a\le A}}
	\|\mathcal{T}^{\lambda}f_{\tau}\|_{L^{p}(B_{K^{2}})}^{p}\Bigr).
	\]
	Here $Gr(k-1,n)$ is the Grassmannian of all $(k-1)$-dimensional subspaces in $\mathbb{R}^{n}$, and $k$ will be specified later. For $U\subset\mathbb{R}^{n}$, set
	\begin{equation}\label{eq norm}
		\|\mathcal{T}^{\lambda}f\|_{BL_{k,A}^{p}(U)}
		\coloneqq
		\Bigl(\sum_{B_{K^{2}}}\frac{|B_{K^{2}}\cap U|}{|B_{K^{2}}|}\,
		\mu_{\mathcal{T}^{\lambda}f}(B_{K^{2}})\Bigr)^{\frac{1}{p}},
	\end{equation}
	where the sum runs over a finitely overlapping collection of $K^{2}$-balls that cover $B_{R}$ and intersect $U$. The quantity in \eqref{eq norm} is called the broad part of $\mathcal{T}^{\lambda}f$.

	By the argument of Proposition 11.1 in \cite{GHI}, it suffices to prove Proposition \ref{prop Broad}. This reduction shows that when $p$ lies in a certain range, the validity of the $k$-broad estimates implies the corresponding estimates for H\"ormander-type oscillatory integral operators at the same exponent $p$. To be more precise, for a fixed integer $k$, if the exponent $p$ satisfies
	\begin{equation}\label{eq p}
		2\frac{2n - k + 2}{2n - k}\le p \le 2\frac{k - 1}{k - 2},
	\end{equation}
	then the $k$-broad estimate implies the desired bound for the oscillatory operator at exponent $p$.
	
	The inequality
	\begin{equation}\label{eq q}
		2\frac{2n - k + 2}{2n - k} \le 2\frac{n + k}{n + k - 2}
	\end{equation}
	holds precisely when $k\le \frac{n + 2}{2}$. Observing that $2\frac{n + k}{n + k - 2}$ decreases monotonically in $k$, the maximal integer $k$ admissible is $\left\lfloor \frac{n + 2}{2}\right\rfloor$. To align with the critical dimension of the Kakeya compression phenomenon, we take $k = \left\lceil \frac{n + 1}{2}\right\rceil$.

	\begin{itemize}
		\item For odd dimensions $n$, taking $k=\lceil\frac{n+1}{2}\rceil$, we have
		\[
		2\frac{3n+3}{3n-1} \le 2\frac{3n+1}{3n-3} - \zeta(n,l),
		\]
		and thus improved broad norm estimates give an overall improvement.
		\item For even dimensions $n$, although we can improve the broad estimate, the exact equality in \eqref{eq q} when $k=\lceil\frac{n+1}{2}\rceil$ precludes any improvement to H\"ormander’s oscillatory integral estimate.
	\end{itemize}
	
	In the following sections, we prove that the next proposition holds for phase functions with finite contact order, thereby completing the proof of Theorem \ref{thm osc}.
	
	\begin{prop}\label{prop Broad}
		Let $\lambda,\,K\ge 1$. For any odd $n$, let $k=(n+1)/2$. For every $\epsilon>0$, there exists $A\in\mathbb{N}$ such that
		\[
		\|\mathcal{T}^{\lambda}f\|_{BL^{p}_{k,A}(\mathbb{R}^{n-1})}
		\lesssim_{\epsilon,K}
		\lambda^{\epsilon}\,\|f\|_{L^{2}}^{2/p}\,\|f\|_{L^{\infty}}^{1-2/p},
		\]
		for every $p$ satisfying 
		\[
		p\ge 2\frac{n+k}{n+k-2}-\zeta(n,l).
		\] Moreover, the implicit constant is at most a bounded power of $K$.
	\end{prop}

	\section{Polynomial partitioning}\label{partitioning}
	In this section, we review preliminaries for polynomial partitioning, including wave packet decomposition and polynomial partitioning algorithms.
	
	\subsection{Wave packet decomposition}
	Given $\varepsilon_\phi\in(0,\varepsilon_0]$. Let $1\le r\le R=\lambda$ and take a collection $\Theta_{r}$ of dyadic cubes of side length $\frac{9}{11}r^{-\frac{1}{2}}$ covering the unit ball ${B}_{\varepsilon_\phi}^{n-1}$. Let $(\eta_{\theta})_{\theta\in\Theta_{r}}$ be a smooth partition of unity on ${B}_{\varepsilon_\phi}^{n-1}$ with $\supp \eta_{\theta}\subset \tfrac{11}{10}\theta$, satisfying
	\[
	\|\partial_{y}^{\alpha}\eta_{\theta}\|_{L^{\infty}}\lesssim_{\alpha}r^{\frac{|\alpha|}{2}}
	\quad\text{for all }\alpha\in\mathbb{N}_{0}^{n-1}.
	\]
	Denote the center of $\theta$ by $y_{\theta}$. Given a function $f$, we perform a Fourier series decomposition of $f\,\eta_{\theta}$ on the region $\tfrac{11}{9}\theta$ and obtain
	\[
	f(y)\,\eta_{\theta}(y)\,\chi_{\frac{11}{10}\theta}(y)
	=\Bigl(\frac{r^{\frac{1}{2}}}{\pi}\Bigr)^{n-1}
	\sum_{v\in r^{\frac{1}{2}}\mathbb{Z}^{n-1}}
	\widehat{f\,\eta_{\theta}}(v)\,e^{2\pi i\,v\cdot y}\,\chi_{\frac{11}{10}\theta}(y).
	\]
	Let $\tilde\eta_{\theta}$ be a nonnegative smooth cutoff with $\supp \tilde\eta_{\theta}\subseteq \tfrac{11}{9}\theta$ and $\tilde\eta_{\theta}\equiv 1$ on $\tfrac{11}{10}\theta$. Then
	\[
	f(y)\,\eta_{\theta}(y)\,\tilde\eta_{\theta}(y)
	=\Bigl(\frac{r^{\frac{1}{2}}}{\pi}\Bigr)^{n-1}
	\sum_{v\in r^{\frac{1}{2}}\mathbb{Z}^{n-1}}
	\widehat{f\,\eta_{\theta}}(v)\,e^{2\pi i\,v\cdot y}\,\tilde\eta_{\theta}(y).
	\]
	Define
	\[
	f_{\theta,v}(y)\coloneqq \Bigl(\frac{r^{\frac{1}{2}}}{\pi}\Bigr)^{n-1}
	\widehat{f\,\eta_{\theta}}(v)\,e^{2\pi i\,v\cdot y}\,\tilde\eta_{\theta}(y).
	\]
	Hence
	\[
	f=\sum_{(\theta,v)\in\Theta_{r}\times r^{\frac{1}{2}}\mathbb{Z}^{n-1}}f_{\theta,v}.
	\]
	
	For any $\omega\in{B}_{\varepsilon_{\phi}}^{n-1}$ and $t\in\mathbb{R}$, define $\Phi(\omega,t;y)$ to be the mapping satisfying
	\[
	\nabla_{y}\phi\bigl(\Phi(\omega,t;y),\,t;\,y\bigr)=\omega,
	\]
	so that $\Phi$ gives the first $n-1$ spatial components of the solution curve. Let $\sigma\coloneqq \epsilon^{C}$, where $C$ is a large constant and $\epsilon$ is as in Theorem \ref{prop Broad}. Define curved $r^{\frac{1}{2}+\sigma}$-tubes by
	\[
	T_{\theta,v}\coloneqq
	\bigl\{(x,t):\,\bigl|x/\lambda-\Phi(v/\lambda,\,t/\lambda;\,y_{\theta})\bigr|\le r^{\frac{1}{2}+\sigma}/\lambda,\;t\in[0,r]\bigr\}.
	\]
	This completes the wave packet decomposition for the ball ${B}_{r}^{n}\subset\mathbb{R}^{n}$. Denote the collection of these wave packets by $\mathbb{T}[{B}_{r}^{n}]$.
	
	Similarly, for $\mathbf{x}_{0}\in {B}_{\lambda}^{n}$, define wave packet decompositions for the ball ${B}_{r}^{n}(\mathbf{x}_{0})$ and write $\mathbb{T}[{B}_{r}^{n}(\mathbf{x}_{0})]$ for that collection. Then for $|\mathbf{x}-\mathbf{x}_{0}|\lesssim r$,
	\[
	\mathcal{T}^{\lambda}f(\mathbf{x})
	=\sum_{T\in \mathbb{T}[{B}_{r}^{n}(\mathbf{x}_{0})]}
	\mathcal{T}^{\lambda}f_{T}(\mathbf{x}).
	\]

	\subsection{Polynomial partitioning algorithms}
	Now, we introduce some notions and definitions that will be used in polynomial partitioning algorithms. 
	
	\begin{definition}[Cells]
		Let $P:\mathbb{R}^{n}\rightarrow \mathbb{R}$ be a non-zero polynomial. Define
		\[
		Z(P)\coloneqq\{z\in \mathbb{R}^{n}: P(z)=0\}.
		\]
		We let $\operatorname{cell}(P)$ denote the collection of all connected components of $\mathbb{R}^{n}\setminus Z(P)$. Each element of $\operatorname{cell}(P)$ will be referred to as a cell of $P$.
	\end{definition}
	
	\begin{definition}[Transverse complete intersection]
		Let $P_{1}, \ldots, P_{n-m}: \mathbb{R}^{n} \rightarrow \mathbb{R}$ be polynomials where $0\le m \le n-1$. Consider the common zero set
		\begin{equation}\label{eq Algebra}
			Z\bigl(P_{1}, \ldots, P_{n-m}\bigr)
			\;:=\;\{x \in \mathbb{R}^{n}: P_{1}(x)=\cdots=P_{n-m}(x)=0\}.
		\end{equation}
		Suppose that for all $z \in Z\bigl(P_{1}, \ldots, P_{n-m}\bigr)$, one has
		\[
		\bigwedge_{j=1}^{n-m} \nabla P_{j}(z) \neq 0.
		\]
		Then any connected branch of this set, or a union of connected branches, is called an $m$-dimensional transverse complete intersection. Given a set $Z$ of the form \eqref{eq Algebra}, the degree of $Z$ is defined by
		\[
		\min\Bigl\{\prod_{i=1}^{n-m} \deg P_{i}\Bigr\},
		\]
		where the minimum is taken over all possible representations of $Z=Z\bigl(P_{1}, \ldots, P_{n-m}\bigr)$.
	\end{definition}

	\begin{lemma}[Polynomial partitioning, Guth \cite{Guth18}, Hickman and Rogers \cite{HR}]
		Fix $r\gg1$ and $d\in\mathbb{N}$, and suppose $F\in L^{1}(\mathbb{R}^{n})$ is non-negative and supported on $B_{r}\cap\mathcal{N}_{r^{1/2+\sigma_{\circ}}}(Z)$ for some $0<\sigma_{\circ}\ll1$, where $B_{r}$ is a ball of radius $r$ in $\mathbb{R}^{n}$ and $Z$ is an $m$-dimensional transverse complete intersection of degree at most $d$. At least one of the following cases holds:
		
		\underline{Cellular case.} There exists a polynomial $P:\mathbb{R}^{n}\rightarrow\mathbb{R}$ of degree $O(d)$ with the following properties:
		\begin{itemize}
			\item[(1)] $\lvert\cell(P)\rvert \sim d^{m}$ and each $O'\in \cell(P)$ has diameter at most $r/d$.
			\item[(2)] If we define 
			\[
			\mathcal{O}
			:=\{\,O'\setminus \mathcal{N}_{r^{1/2+\sigma_{\circ}}}(Z): O'\in \cell(P)\},
			\]
			then 
			\[
			\int_{O}F \sim d^{-m}\int_{\mathbb{R}^{n}}F
			\quad\text{for all }O\in\mathcal{O}.
			\]
		\end{itemize}
		
		\underline{Algebraic case.} There exists an $(m-1)$-dimensional transverse complete intersection $Y$ of degree at most $O(d)$ such that
		\[
		\int_{B_{r}\cap \mathcal{N}_{r^{1/2+\sigma_{\circ}}}(Z)}F
		\;\lesssim\;
		\int_{B_{r}\cap \mathcal{N}_{r^{1/2+\sigma_{\circ}}}(Y)}F.
		\]
	\end{lemma}

	We revisit the polynomial partitioning algorithm from \cite{GWZ22}, modified from \cite{HR}. We define a sequence of small parameters $\{\sigma_i\}_{i=0}^n$ satisfying
	\[
	\sigma \ll \sigma_{n} \ll \sigma_{n-1} \ll \cdots \ll \sigma_{i} \ll \cdots \ll \sigma_{1} \ll \sigma_{0} \ll \epsilon,
	\]
	where $0 \le i \le n$. For example, these $\sigma$ and $\sigma_i$ can be taken as constant powers of $\epsilon$.
	
	We then partition ${B}_R^n$ into a finitely overlapping collection of balls $\{B_{\nu}\}_{\nu}$, where each $B_{\nu}$ has radius $R^{1-\sigma}$. The algorithm is described in detail in \cite{GWZ22}. Below, we recall its outputs.
	
	\begin{enumerate}
		\item[\textbf{Output 1:}] We obtain a sequence of nodes
		\[
		\mathbf{n}_{0}^{\star}, \mathbf{n}_{1}^{\star}, \dots, \mathbf{n}_{l_{0}}^{\star},
		\]
		where for each $0 \le l \le l_{0}$ (with $l_{0} \in \mathbb{N}$), $\mathbf{n}_{l}^{\star}$ represents a collection of open sets in $\mathbb{R}^{n}$. Each set $\mathbf{n}_{l}^{\star}$ is characterized by two parameters:
		\begin{itemize}
			\item A dimensional parameter $\dim (\mathbf{n}_{l}^{\star})$.
			\item A radius parameter $\rho (\mathbf{n}_{l}^{\star})$.
		\end{itemize}
		Moreover, let $m := \dim(\mathbf{n}_{l_{0}}^{\star})$ with $m \ge k$. The dimensional parameter $\dim(\mathbf{n}_{l}^{\star})$ is non-increasing in $l$.
		
		\item[\textbf{Output 2:}] We extract the key components from the nodes. A node $\mathbf{n}_{l}^{\star}$ is called an \textit{R-child} if
		\[
		\dim(\mathbf{n}_{l}^{\star}) = \dim(\mathbf{n}_{l-1}^{\star}) - 1.
		\]
		We collect all R-child nodes and denote them by
		\[
		\mathcal{G}_{n}, \mathcal{G}_{n-1}, \dots, \mathcal{G}_{m},
		\]
		where $\mathcal{G}_{n} = \mathbf{n}_{0}^{\star}$. Each element of $\mathcal{G}_{n'}$ has the form
		\[
		B_{r_{n'}} \cap \mathcal{N}_{r_{n'}^{\frac{1}{2} + \sigma_{n'}}}(S_{n'}),
		\]
		where:
		\begin{itemize}
			\item The node $\textbf{n}_{0}^*$ consists of only one element ${B}_{R}^{n}$.
			\item $S_{n'}$ is an algebraic variety of dimension $n'$.
			\item $\mathcal{N}_{r_{n'}^{\frac{1}{2} + \sigma_{n'}}}(S_{n'})$ denotes the $r_{n'}^{\frac{1}{2}+\sigma_{n'}}$-neighborhood of $S_{n'}$.
			\item $\dim(\mathcal{G}_{n'}) = n'$ for all $m \le n' \le n$.
			\item $r_{n'} := \rho(\mathcal{G}_{n'})$ for all $m \le n' \le n$.
			\item $r_{m-1} := 1$.
		\end{itemize}
		From now on, for convenience, we denote
		\[
		\mathcal{S}_{n'} := B_{r_{n'}} \cap \mathcal{N}_{r_{n'}^{\frac{1}{2}+\sigma_{n'}}}(S_{n'}).
		\]
		
		\item[\textbf{Output 3:}] Each open set $O \in \mathbf{n}_{l_{0}}^{\star}$ satisfies
		\[
		\mathrm{diam} (O) \le R^{\sigma_{0}},
		\]
		which serves as the stopping condition for the algorithm. For each such $O$, we define an associated function $f_{\nu, O}$ by
		\[
		f_{\nu, O} = \sum_{T \in \mathbb{T}'[B_{\rho(\mathbf{n}_{l_{0}}^{\star})}]} f_{T},
		\]
		where:
		\begin{itemize}
			\item $B_{\rho(\mathbf{n}_{l_{0}}^{\star})} \subset \mathbb{R}^{n}$ is the ball of radius $\rho(\mathbf{n}_{l_{0}}^{\star})$ containing $O$.
			\item $\mathbb{T}'[B_{\rho(\mathbf{n}_{l_{0}}^{\star})}]$ is a subcollection of $\mathbb{T}[B_{\rho(\mathbf{n}_{l_{0}}^{\star})}]$.
		\end{itemize}
		
		\item[\textbf{Output 4:}] For each $\nu$ and each $\mathcal{S}_{n'} \in \mathcal{G}_{n'}$ satisfying $\mathcal{S}_{n'} \cap B_{\nu} \neq \varnothing$, we define
		\[
		f_{\nu, \mathcal{S}_{n'}}^{\star} = \sum_{T \in \mathbb{T}[B_{r_{n'}}]} f_{T}.
		\]
		We use $\mathbb{T}[B_{r_{n'}}]$ to denote the collection of tubes that satisfy the following properties:
		\begin{itemize}
			\item $|T_{1} \cap T_{2}| \leq \frac{1}{2} |T_{1}|$ for any $T_{1}, T_{2} \in \mathbb{T}[B_{r_{n'}}]$;
			\item each tube has length $r_{n'}$ and width $r_{n'}^{\frac{1}{2} + \sigma_{n'}}$;
			\item each tube is contained in $\mathcal{S}_{n'}$.
		\end{itemize}
		\item[\textbf{Output 5:}] We have parameters $\{D_i\}$ that are integer powers of $d$ and satisfy
		\begin{equation}\label{eq output 5}
			D_{i} \le \frac{r_{i+1}}{r_{i}} \quad \text{for all } i < n,
		\end{equation}
		with the additional condition $D_{n} = 1$.
	\end{enumerate}
	We remark that \eqref{eq output 5} in \textbf{Output 5} is slightly different than that stated in \cite{GWZ22}; however, this condition still holds by Lemma 5.10 in \cite{GWZ22} and is a better version for us to use.
	
	At this stage, we introduce some notation that will help describe the properties of the outputs given above. Let $\{p_{n'}\}_{n'=m}^n$ satisfy
	\[
	p_{m} \ge p_{m+1} \ge \cdots \ge p_{n} = p \ge 2,
	\]
	where $m \ge k$, and $k$, $p$ are the exponents in Theorem \ref{prop Broad}. For each $m+1 \le n' \le n-1$, define $\alpha_{n'}$, $\beta_{n'} \in [0,1]$ by
	\[
	\frac{1}{p_{n'}} \;=\; \frac{1-\alpha_{n'-1}}{2} \;+\; \frac{\alpha_{n'-1}}{p_{n'-1}},
	\qquad
	\beta_{n'} \;=\; \prod_{i=n'}^{n-1} \alpha_{i},
	\]
	and set $\alpha_{n} = \beta_{n} = 1$. Then, using Properties 1-3 in Section 5 of \cite{GWZ22}, the outputs satisfy the following two propositions.
	\begin{prop}\label{prop Property 1}
		\[
		\|\mathcal{T}^{\lambda}f\|_{BL^{p}_{k,A}(B_{R})} 
		\lesssim_{\epsilon}R^{O(\sigma)} 
		\prod_{i=m-1}^{n-1} r_{i}^{\frac{\beta_{i+1}-\beta_{i}}{2}} 
		D_{i}^{\frac{\beta_{i+1}}{2} - \bigl(\frac{1}{2} - \frac{1}{p_{n}}\bigr)} 
		\|f\|_{2}^{\frac{2}{p_{n}}} 
		\max_{O \in \mathbf{n}_{l_{0}}^{\star}} 
		\|f_{\nu,O}\|_{2}^{1 - \frac{2}{p_{n}}},
		\]
		where $\nu$ denotes the ball $B_{\nu} \subset B_{R}$ containing $O$ and of radius $R^{1-\sigma}$, and $m \ge k$.
	\end{prop}
	Proposition \ref{prop Property 1} follows from Property 1 together with repeated applications of Property 2 in \cite{GWZ22}.
	
	\begin{prop}\label{prop Property 3}
		For $n' \le n-1$, one has
		\[
		\max_{\mathcal{S}_{n'} \in \mathcal{G}_{n'}} 
		\|f_{\nu,\mathcal{S}_{n'}}^{\star}\|_{2}^{2} 
		\lesssim_{\epsilon}R^{O(\sigma)} 
		\Bigl(\frac{r_{n'+1}}{r_{n'}}\Bigr)^{-\frac{n-n'-1}{2}} 
		D_{n'}^{-n' + \sigma} 
		\max_{\mathcal{S}_{n'+1} \in \mathcal{G}_{n'+1}} 
		\|f_{\nu,\mathcal{S}_{n'+1}}^{\star}\|_{2}^{2},
		\]
		\[
		\max_{\mathcal{S}_{n'} \in \mathcal{G}_{n'}} 
		\max_{\theta} 
		\|f_{\nu,\mathcal{S}_{n'}}^{\star}\|_{L_{\mathrm{avg}}^{2}(\theta)}^{2} 
		\lesssim_{\epsilon}R^{O(\sigma)} 
		\Bigl(\frac{r_{n'+1}}{r_{n'}}\Bigr)^{-\frac{n-n'-1}{2}} 
		D_{n'}^{\sigma} 
		\max_{\mathcal{S}_{n'+1} \in \mathcal{G}_{n'+1}} 
		\max_{\theta} 
		\|f_{\nu,\mathcal{S}_{n'+1}}^{\star}\|_{L_{\mathrm{avg}}^{2}(\theta)}^{2},
		\]
		where $\theta$ is a ball of radius $\rho^{-\frac{1}{2}}$ in frequency space, with $1 \le \rho \le r_{n'}$.
	\end{prop}
	
	\section{Contact order and polynomial Wolff axioms in $\mathbb{R}^n$}\label{contact order}
	
	In this section, we generalize the contact order condition from $\mathbb{R}^3$ \cite{DGGZ24} to arbitrary dimensions and prove that the family of tubes induced by a phase function with finite contact order at the origin satisfies the polynomial Wolff axioms.
	
	Let $\phi$ be a H\"ormander phase, and $\varepsilon_\phi\in(0,\varepsilon_0]$. For any point $(x_0,t_0,y_0)\in {B}_{\varepsilon_{\phi}}^{n-1}\times {B}_{\varepsilon_{\phi}}^{1}\times {B}_{\varepsilon_{\phi}}^{n-1}$, H\"ormander’s non-degeneracy condition guarantees the existence and uniqueness of a curve $t\mapsto (X_0(t),\,t)\in {B}_{\varepsilon_{\phi}}^{n-1}\times {B}_{\varepsilon_{\phi}}$ satisfying 
	\[
	\nabla_{y}\phi\bigl(X_0(t)+x_0,\;t+t_0;\;y_0\bigr)
	=\nabla_{y}\phi\bigl(x_0,\;t_0;\;y_0\bigr).
	\]
	We normalize $\phi$ at $(x_0,t_0;y_0)$ by setting 
	\[
	\phi_0(x,t;y)=\phi\bigl(x+x_0,\;t+t_0;\;y+y_0\bigr)
	-\phi\bigl(x_0,\;t_0;\;y+y_0\bigr).
	\]
	In ${B}_{\varepsilon_{\phi}}^{n}$, define 
	\[
	D_{i,j}(t;x_0,t_0;y_0)
	=\partial_{y_i}\partial_{y_j}\,\phi_0\bigl(X_0(t),\,t;\,0\bigr),
	\quad 1\le i,j\le n-1,
	\]
	and abbreviate $D_{i,j}(t)$ when no confusion arises. Let the matrix 
	\[
	D(t) = \big( D_{i,j}(t) \big)_{1 \le i,j \le n-1},
	\]
	and denote its determinant by $\textbf{D}(t)$.
	For integers $1\le k\le n-2$ and indices $i_1,\dots,i_k$, $j_1,\dots,j_k$, write 
	\[
	\bigl(D_{i_1,j_1}D_{i_2,j_2}\cdots D_{i_k,j_k}\bigr)^{(h)}(0)
	\]
	for the $h$th derivative in $t$ of $D_{i_1,j_1}\cdots D_{i_k,j_k}$ at $t=0$.  Consider the matrix $\tilde D$ whose rows consist of all tuples 
	\[
	\bigl((D_{i_1,j_1}\cdots D_{i_k,j_k})^{(h)}(0)\bigr)_{1\le h\le l},
	\quad 1\le k\le n-2,
	\]
	with indices satisfying $i_\alpha\neq i_\beta$ and $j_\alpha\neq j_\beta$ for $\alpha\neq\beta$, together with the row $(\textbf{D}^{(h)}(0))_{1\le h\le l}$.  Since $D(t)$ is symmetric, $\tilde D$ contains duplicate rows; remove all duplicates to obtain $\mathcal D$.
	
	For any $1\leq h \leq l$, the $h$th column of $\mathcal D$ consists of the term $\textbf{D}^{(h)}(0)$ and all distinct terms of the form
	\[
	\left.\frac{\partial^h}{\partial t^h}\Bigl(\prod_{\alpha=1}^k D_{\,i_\alpha,j_\alpha}\Bigr)\right|_{t=0},
	\quad 1\le k\le n-2,
	\]
	subject to
	\begin{itemize}
		\item $k=1$: the single entries $D_{i,j}$ (all distinct).
		\item $k=2$: products $D_{i_1,j_1}D_{i_2,j_2}$ with $i_1\neq i_2$ and $j_1\neq j_2$ (all distinct).
		\item In general: products $\prod_{\alpha=1}^k D_{\,i_\alpha,j_\alpha}$ where
		\[
		i_\alpha\neq i_\beta
		\quad\text{and}\quad
		j_\alpha\neq j_\beta
		\quad\text{for all }\alpha\neq\beta.
		\]
	\end{itemize}
	
	We note that the matrix $\mathcal{D}(t)$ excludes elements of the form $(D_{i_{1},j_{1}}\cdots D_{i_{n-1},j_{n-1}})^{(h)}(0)$ for $1\le h\le l$.
	
	Let $A(n)$ be an integer depending on the dimension $n$ given by
	\[
	A(n)=1 + \sum_{k=1}^{n-2} \Biggl[ 
	\binom{n-1}{k} \frac{k! + I(k)}{2} 
	+ \sum_{i=0}^{k-1} 
	\frac{\binom{n-1}{i}\binom{n-1-i}{\,k-i\,}\binom{n-1-k}{\,k-i\,}}{2} 
	\Bigl( k! - (k-i)! \,\frac{i! - I(i)}{2} \Bigr) 
	\Biggr],
	\]
	where $I(k)$ is the number of involutions in the symmetric group $S(k)$. The row count of the new matrix $\mathcal{D}(t)$ is exactly given by $A(n)$. For a proof, see Proposition \ref{prop scalingPWA} in the Appendix.
	In $\mathbb{R}^{n}$, since the row count of $\mathcal{D}(t)$ is $A(n)$, we only need to consider integers $l \ge A(n)$. We refer to $A(n)$ as the critical contact order in $\mathbb{R}^{n}$, which is the smallest admissible contact order.
	
	\begin{definition}[Contact order condition]
		If there exists an integer $l \ge A(n)$ such that the matrix $\mathcal{D}$ has row rank $A(n)$, then we say that the phase $\phi$ has contact order $\le l$ at the point $(x_{0},t_{0};y_{0})$.
	\end{definition}
	
	When $n=3$, we have $A(n)=4$ and the matrix $\mathcal{D}$ is
	\begin{equation}\label{eq ct3}
		\begin{pmatrix}
			\textbf{D}^{(1)}(0)      & \cdots  & \textbf{D}^{(l)}(0) \\
			D_{11}^{(1)}(0) & \cdots & D_{11}^{(l)}(0) \\
			D_{12}^{(1)}(0) & \cdots & D_{12}^{(l)}(0) \\
			D_{22}^{(1)}(0) & \cdots & D_{22}^{(l)}(0) \\
		\end{pmatrix}.
	\end{equation}
	If there exists an integer $l \ge 4$ such that the matrix $\mathcal{D}$ has rank 4, then we say the phase $\phi$ has contact order $l$ at the point $(x_{0},t_{0})$. This matrix \eqref{eq ct3} appeared in \cite{Bourgain91} and corresponds to the contact order condition in dimension 3 introduced by Dai, Gong, the first author, and Zhang \cite{DGGZ24}. They also showed that the Riemannian distance function of manifolds $M^{3}$ with Sogge’s chaotic curvature \cite{Sogge99} has a contact order $\le 4$ at the origin.
	
	For phase functions $\phi$ with contact order $\le l$ at the origin, we establish that the associated families of tubes in different directions satisfy the polynomial Wolff axioms introduced in \cite{KatzR}. The polynomial Wolff axiom is a powerful tool to counteract the Kakeya compression phenomenon. The following theorem is precisely where we determine $ \varepsilon_\phi $ for a given phase function $ \phi $.
	
	\begin{theorem}\label{thm PWA}
		Let $n \ge 3$ and $l\ge A(n)$. Suppose $\phi$ has a finite contact order at most $l$ at the origin $(0,0;0)$. Then there exists a constant $\varepsilon_\phi\in(0,\varepsilon_0]$ such that the following polynomial Wolff axiom holds:
		Let $S$ be a semi-algebraic set with complexity at most $E$, and let $\mathbb{T}$ be a direction-separated family associated with $(\phi,\varepsilon_{\phi})$ (see Definition \ref{dsf}). For any $\epsilon > 0$, there exists a constant $C(n,E,\epsilon)$ such that
		\[
		\#\{T \in \mathbb{T} : T \subset S\} \;\le\; 
		C(n,E,\epsilon)\,\mathcal{L}^{n}(S)\,\delta^{1 - n - \epsilon},
		\]
		where $\mathcal{L}^{n}(S)$ denotes the $n$-dimensional Lebesgue measure of $S$.
	\end{theorem}
	\begin{proof}
		First, let us fix some notation. We denote the curve pointing in direction $y$ and passing through $v$ as $(\Phi(v,t;y),t)$, i.e.,
		\[
		\nabla_{y}\phi(\Phi(v,t;y),t;y)=v.
		\]
		Using the standard reduction in \cite[Theorem 3.1]{GWZ22}, it suffices to show that 
		\[
		\int_{|t|\le \varepsilon_{\phi}}\bigl|\det\bigl(\nabla_{v}\Phi(v,t;y)\cdot M+\nabla_{y}\Phi(v,t;y)\bigr)\bigr|\,dt\gtrsim1,
		\]
		for all $|v|\le\varepsilon_{\phi}$, $|y|\le\varepsilon_{\phi}$ and an arbitrary matrix $M$.
		
		Take derivatives in $v$ on the equation above; then we have 
		\[
		\nabla_{x}\nabla_{y}\phi\bigl(\Phi(v,t;y),t;y\bigr)\,\nabla_{v}\Phi(v,t;y)=I_{(n-1)\times(n-1)}.
		\]
		Taking derivatives in $y$ on this equation gives 
		\[
		\nabla_{x}\nabla_{y}\phi\bigl(\Phi(v,t;y),t;y\bigr)\,\nabla_{y}\Phi(v,t;y)
		+\nabla_{y}^{2}\phi\bigl(\Phi(v,t;y),t;y\bigr)=\mathbf{0},
		\]
		where $\mathbf{0}$ denotes the $(n-1)\times(n-1)$ zero matrix.
		
		Since $\phi$ satisfies H\"ormander’s non-degeneracy condition, $\nabla_{y}^{2}\phi\bigl(\Phi(v,t;y),t;y\bigr)$ is a non-degenerate matrix. By combining the two preceding equations, we can reduce the inequality above to  
		\[
		\int_{|t|\le \varepsilon_{\phi}}\bigl|\det\bigl(M+\nabla_{y}^{2}\phi(\Phi(v,t;y),t;y)\bigr)\bigr|\,dt
		\gtrsim1.
		\]
		We only need to prove that this holds for any $v,y\in{B}_{\varepsilon_{\phi}}^{n-1}$ and any matrix $M$.
		
		Consider the integrand
		\[
		\det\bigl(M+\nabla_{y}^{2}\phi(\Phi(v,t;y),t;y)\bigr)
		=\det\bigl(U+\nabla_{y}^{2}\phi_{0}(\Phi(v,t;y),t;y)\bigr),
		\]
		where 
		\[
		U=M+\nabla_{y}^{2}\phi(\Phi(v,0;y),0;y).
		\]
		Since $M$ is arbitrary, $U$ is also arbitrary (it depends on $M,v,y$). Write 
		\[
		U=\begin{pmatrix}
			u_{1,1}&u_{1,2}&\cdots&u_{1,n-1}\\
			u_{2,1}&u_{2,2}&\cdots&u_{2,n-1}\\
			\vdots&\vdots&\ddots&\vdots\\
			u_{n-1,1}&u_{n-1,2}&\cdots&u_{n-1,n-1}
		\end{pmatrix}.
		\]
		
		{ The constant $\varepsilon_{\phi}$ can be chosen sufficiently small such that the phase $\phi$ has contact order $\le l$ for any point $(x,t;y)\in {B}_{\epsilon_{\phi}}^{n-1}\times {B}_{\varepsilon_{\phi}}^{1}\times{B}_{\epsilon_{\phi}}^{n-1}$, owing to the stability of full row rank under small perturbations.} Now we only analyze the contact order at $(0,0;y)$.
		For notational simplicity, we still write $D_{i,j}(t)$ for $D_{i,j}(t;0,0;y)$ throughout this analysis.
		
		By the definition of $D$, it is equivalent to consider the matrix
		\[
		U+D
		=\begin{pmatrix}
			D_{1,1}(t)+u_{1,1}&D_{1,2}(t)+u_{1,2}&\cdots&D_{1,n-1}(t)+u_{1,n-1}\\
			D_{2,1}(t)+u_{2,1}&D_{2,2}(t)+u_{2,2}&\cdots&D_{2,n-1}(t)+u_{2,n-1}\\
			\vdots&\vdots&\ddots&\vdots\\
			D_{n-1,1}(t)+u_{n-1,1}&D_{n-1,2}(t)+u_{n-1,2}&\cdots&D_{n-1,n-1}(t)+u_{n-1,n-1}
		\end{pmatrix}.
		\]
		
		Define the basis‐vector function
		\[
		\mathbf{F}(t)
		=\bigl(D_{i,j}(t),\;\prod_{\alpha=1}^{2}D_{i_\alpha,j_\alpha}(t),\;\dots,\;\textbf{D}(t)\bigr),
		\]
		where the products range over all indices satisfying $i_{\alpha}\neq i_{\beta}$ and $j_{\alpha}\neq j_{\beta}$ for $\alpha\neq\beta$, excluding duplicates. The last component of $\mathbf{F}(t)$ is $\textbf{D}(t)$, while the first $A(n)-1$ components are all possible products
		\[
		\prod_{\alpha=1}^{k}D_{i_\alpha,j_\alpha}(t),
		\quad
		1\le k\le n-2,
		\quad
		i_{\alpha}\neq i_{\beta},\;j_{\alpha}\neq j_{\beta}\;\forall\,\alpha\neq\beta.
		\]
		
		We expand the determinant:
		\[
		\det(U+D)
		=\det(U)+(C_{1},C_{2},\dots,\mathbf C_{A(n)-1},1)\cdot \mathbf{F}(t)^{T},
		\]
		where $\mathbf{F}(t)^{T}$ is the transpose of $\mathbf{F}(t)$, and $\{C_{k}\}$ ($1\le k\le A(n)-1$) are coefficients determined by $U$.
		
		Since $\mathbf{F}(t)$ is smooth, perform a Taylor expansion up to the $l$-th order on each component of $\mathbf{F}(t)$:
		\[
		\det(U+D)
		=\det(U)
		+(C_{1},C_{2},\dots,\mathbf C_{A(n)-1},1)\,\mathcal{D}\,\bigl(t,\tfrac{t^{2}}{2!},\dots,\tfrac{t^{l}}{l!}\bigr)^{T}
		+o(t^{l}).
		\]
		Denote 
		\[
		(H_{1},H_{2},\dots,H_{l})
		:=(C_{1},C_{2},\dots,\mathbf C_{A(n)-1},1)\,\mathcal{D}.
		\]
		Then we can write
		\[
		W(t)\coloneqq\det\bigl(M+\nabla_{y}^{2}\phi(\Phi(v,t;y),t;y)\bigr)
		=\det(U)
		+\sum_{k=1}^{l}H_{k}\,\frac{t^{k}}{k!}
		+o(t^{l}).
		\]
		
		Since the matrix $\mathcal{D}$ has full row rank $l$, we have
		\[
		\|(H_{1},H_{2},\dots,H_{l})\|_{1}
		\gtrsim1+\|(C_{1},C_{2},\dots,\mathbf C_{A(n)-1})\|_{1}.
		\]
		By the pigeonhole principle, there exists an integer $1\le k'\le l$ such that $|W^{k'}(0)|\gtrsim1$, so there must exist a constant‐measure neighborhood where $|W(t)|>0$. Hence
		\[
		\int_{|t|\le\varepsilon_{\phi}}|W(t)|\,dt\gtrsim1.
		\]
		This completes the proof of Theorem \ref{thm PWA}.
	\end{proof}
	
	\section{Proof of Proposition \ref{prop Broad} via polynomial Wolff axioms}\label{proof}
	This section is devoted to the proof of Proposition \ref{prop Broad}. The key ingredients include a rescaled version of the polynomial Wolff axiom and polynomial partitioning.
	
	We first introduce some notation. By polynomial partitioning, the curved tubes in $\mathbb{T}[B_{r_{n'}}]$ are contained in $\mathcal{S}_{n'}$ for any $m+1 \le n' \le n-1$. Without loss of generality, assume that the semialgebraic set $\mathcal{S}_{n'}$ is centered at the origin. Thus, we consider tubes of length $r_{n'}$. We define
	\[
	T_{\theta,v} := \left\{ (x,t) : \left|\frac{x}{\lambda} - \Phi\left(\frac{v}{\lambda}, \frac{t}{\lambda}; y_{\theta}\right)\right| \leq \frac{r_{n'}^{\frac{1}{2}+\sigma}}{\lambda}, \ t \in [0, r_{n'}] \right\},
	\]
	where $y_{\theta}$ is the center of $\theta$. Since the phase $\phi$ determines a constant $\varepsilon_\phi$, fixed throughout this section (see Theorem \ref{thm PWA}), we, for simplicity, treat $\varepsilon_\phi$ as unit scale. Accordingly, we rescale $r_{n'} \in (0,\varepsilon_\phi]$ to unit length.
	\[
	T_{\theta,v}' := \left\{ (x,t) : \left|\frac{r_{n'}x}{\lambda} - \Phi\left(\frac{r_{n'}v}{\lambda}, \frac{r_{n'}t}{\lambda}; y_{\theta}\right)\right| \leq \frac{r_{n'}^{\frac{1}{2}+\sigma}}{\lambda}, \ t \in [0,1] \right\}.
	\]
	Let $\lambda_{n'} = \frac{\lambda}{r_{n'}}$. Then equivalently,
	\[
	T_{\theta,v}' = \left\{ (x,t) : \left|x - \lambda_{n'}\Phi\left(\frac{v}{\lambda_{n'}}, \frac{t}{\lambda_{n'}}; y_{\theta}\right)\right| \leq r_{n'}^{-\frac{1}{2}+\sigma}, \ t \in [0,1] \right\}.
	\]
	We denote
	\[
	\Phi_{\lambda_{n'}}(v,t;y) := \lambda_{n'}\Phi\left(\frac{v}{\lambda_{n'}}, \frac{t}{\lambda_{n'}}; y\right).
	\]
	It is straightforward to verify that
	\begin{equation}\label{eq scalingcurve}
		\frac{v}{\lambda_{n'}} 
		= \nabla_{y}\phi_{\lambda_{n'}}\bigl(\Phi_{\lambda_{n'}}(v,t;y),\,t;\,y\bigr),
	\end{equation}
	where
	\begin{equation}\label{eq scalingphi}
		\phi_{\lambda_{n'}}(x,t;y) := \lambda_{n'}\phi\bigl(\tfrac{x}{\lambda_{n'}}, \tfrac{t}{\lambda_{n'}}; y\bigr).   
	\end{equation}
	
	For the rescaled phase $\phi_{\lambda_{n'}}$, we denote wave packets of unit length and width $\kappa$ by
	\[
	T_{y_{\theta},v,\Phi_{\lambda_{n'}}}(\kappa,1) := \left\{ (x,t) : \left|x - \Phi_{\lambda_{n'}}(v,t; y_{\theta})\right| \leq \kappa, \ t \in [0,1] \right\},
	\]
	{where $\theta$ denotes the square on $\mathbb{R}^{n-1}$ of side length $\kappa$ and $y_{\theta}$ denotes its center.}
	We have the following estimate, which can be regarded as a polynomial Wolff axiom for the rescaled phase.
	
	\begin{prop}\label{prop scalingPWA}
		Let $S \subset {B}^{n}$ be a semi-algebraic set of complexity at most $E$. Let 
		\[
		\mathbb{T} = \{\,T_{y_{\theta},v,\Phi_{\lambda_{n'}}}(\kappa,1)\}
		\]
		be a collection of tubes pointing in different directions. Then for any $\epsilon > 0$, there exists a constant $C_{\epsilon,E}$ such that
		\[
		\#\{\,T\in\mathbb{T}:T \subset S\}
		\le
		C_{\epsilon,E}\,\min\bigl\{\mathcal{L}^{n}(S)\,\lambda_{n'}^{l-(n-1)},\,
		\mathcal{L}^{n}\bigl(\mathcal{N}_{\tfrac{1}{\lambda_{n'}}}(S)\bigr)\bigr\}\,\kappa^{-(n-1)-\epsilon},
		\]
		where
		\begin{itemize}
			\item $\phi$ is a phase function with contact order at most $l$,
			\item $\mathcal{N}_{\tfrac{1}{\lambda_{n'}}}(S)$ denotes the $\tfrac{1}{\lambda_{n'}}$-neighborhood of $S$.
		\end{itemize}
	\end{prop}

	\begin{proof}
		Let us first prove the upper bound
		\[
		\#\{T\in \mathbb{T}:T\subset S\}\le C_{\epsilon,E}\,\mathcal{L}^n(S)\,\lambda_{n'}^{l-(n-1)}\,\kappa^{-(n-1)-\epsilon}.
		\]
		By the argument in \cite[Section 3]{GWZ22}, it suffices to show
		\begin{equation}\label{eq reduce}
			\int_{|t|\le\varepsilon_{\phi}}\Bigl|\det\bigl(\nabla_{v}\Phi_{\lambda_{n'}}(v,t;y)\cdot M+\nabla_{y}\Phi_{\lambda_{n'}}(v,t;y)\bigr)\Bigr|\,dt\gtrsim\lambda_{n'}^{n-1-l},
		\end{equation}
		uniformly in $v\in{B}_{\varepsilon_{\phi}}^{n-1}$, $y\in{B}_{\varepsilon_{\phi}}^{n-1}$, and all $(n-1)\times(n-1)$ matrices $M$. Combining this with \eqref{eq scalingcurve} and \eqref{eq scalingphi}, it is equivalent to
		\begin{equation}\label{eq reduce2}
			\int_{|t|\le\varepsilon_{\phi}}\bigl|\det\bigl(M+\nabla_{y}^{2}\phi_{\lambda_{n'}}\bigl(\Phi_{\lambda_{n'}}(v,t;y),t;y\bigr)\bigr)\bigr|\,dt\gtrsim\lambda_{n'}^{n-1-l},
		\end{equation}
		under the same uniformity conditions. Without loss of generality, it suffices to prove \eqref{eq reduce} for $v=0$ and $y=0$, since the contact order condition is stable under small perturbations. Using $\Phi_{\lambda_{n'}}(0,t;0)=0$ for $|t|\le\varepsilon_{\phi}$, \eqref{eq reduce2} simplifies to
		\[
		\int_{|t|\le\varepsilon_{\phi}}\bigl|\det\bigl(M+\nabla_{y}^{2}\phi_{\lambda_{n'}}(0,t;0)\bigr)\bigr|\,dt\gtrsim\lambda_{n'}^{n-1-l}.
		\]
		Since 
		\[
		\nabla_{y}^{2}\phi_{\lambda_{n'}}(0,t;0)
		=\lambda_{n'}\,\nabla_{y}^{2}\phi\Bigl(0,\tfrac{t}{\lambda_{n'}};0\Bigr),
		\]
		we compute:
		\begin{align*}
			\det\bigl(M+\nabla_{y}^{2}\phi_{\lambda_{n'}}(0,t;0)\bigr)
			&=\det\Bigl(M+\lambda_{n'}\nabla_{y}^{2}\phi\bigl(0,\tfrac{t}{\lambda_{n'}};0\bigr)\Bigr)\\
			&=\det\Bigl(M+\lambda_{n'}\nabla_{y}^{2}\phi(0,0;0)
			+\lambda_{n'}\nabla_{y}^{2}\phi_{0}\bigl(0,\tfrac{t}{\lambda_{n'}},0\bigr)\Bigr)\\
			&=\lambda_{n'}^{n-1}\Bigl(1+\sum_{i=1}^{l}c_{i}\bigl(\tfrac{t}{\lambda_{n'}}\bigr)^{i}+o\bigl(t^{l}\bigr)\Bigr)\\
			&\ge\lambda_{n'}^{n-1-l}\Bigl(1+\sum_{i=1}^{l}c_{i}\,t^{i}+o\bigl(t^{l}\bigr)\Bigr)
			\gtrsim\lambda_{n'}^{n-1-l}.
		\end{align*}
		This proves the first upper bound.
		
		Now we prove the second upper bound. By the standard reduction in \cite{Hor73}, it suffices to consider phases of the form
		\[
		\phi(x,t;y)=x\cdot y+t\,y^{T}Ay+\psi(x,t;y),
		\]
		where 
		\[
		\psi(x,t;y)=O\bigl(|t|\,|y|^{3}+|(x,t)|^{2}\,|y|^{2}\bigr).
		\]
		Write $\psi(x,t;y)=t\,Q(y)+\tilde{\psi}(x,t;y)$, where $Q(y)=O(|y|^{3})$ and 
		\[
		\tilde{\psi}(x,t;y)=O\bigl(|(x,t)|^{2}\,|y|^{2}\bigr).
		\]
		Then the rescaled phase is
		\[
		\phi_{\lambda_{n'}}(x,t;y)
		=x\cdot y+t\,y^{T}Ay+t\,Q(y)
		+\lambda_{n'}\,\tilde{\psi}\Bigl(\tfrac{x}{\lambda_{n'}},\tfrac{t}{\lambda_{n'}};y\Bigr).
		\]
		Define 
		\[
		\tilde{\phi}_{\lambda_{n'}}(x,t;y)=x\cdot y+t\,y^{T}Ay+t\,Q(y).
		\]
		For $\kappa<1$, set
		\[
		\tilde{T}_{y_{\theta},v,\Phi_{\lambda_{n'}}}(\kappa,1)
		=\bigl\{(x,t):\;|\nabla_{y}\tilde{\phi}_{\lambda_{n'}}(x,t;y_{\theta})-v|\le\kappa,\;t\in[0,1]\bigr\}.
		\]
		If $(x,t)$ satisfies
		\[
		\nabla_{y}\phi_{\lambda_{n'}}(x,t;y_{\theta})=v
		\quad\text{and}\quad
		(x',t)\text{ satisfies }
		\nabla_{y}\tilde{\phi}_{\lambda_{n'}}(x',t;y_{\theta})=v,
		\]
		then
		\[
		\nabla_{y}\phi_{\lambda_{n'}}(x,t;y_{\theta})
		-\nabla_{y}\phi_{\lambda_{n'}}(x',t;y_{\theta})
		+\nabla_{y}\phi_{\lambda_{n'}}(x',t;y_{\theta})
		-\nabla_{y}\tilde{\phi}_{\lambda_{n'}}(x',t;y_{\theta})=0.
		\]
		Since $\phi_{\lambda_{n'}}$ satisfies H\"ormander’s nondegeneracy, 
		\[
		|\nabla_{y}\phi_{\lambda_{n'}}(x,t;y_{\theta})
		-\nabla_{y}\phi_{\lambda_{n'}}(x',t;y_{\theta})|\sim|x-x'|,
		\]
		and
		\[
		|\nabla_{y}\phi_{\lambda_{n'}}(x',t;y_{\theta})
		-\nabla_{y}\tilde{\phi}_{\lambda_{n'}}(x',t;y_{\theta})|
		\lesssim\lambda_{n'}\cdot\tfrac{1}{\lambda_{n'}^{2}}=\tfrac{1}{\lambda_{n'}}.
		\]
		Hence for any 
		$T\in T_{y_{\theta},v,\Phi_{\lambda_{n'}}}(\kappa,1)$, there is 
		$\tilde{T}\in\tilde{T}_{y_{\theta},v,\Phi_{\lambda_{n'}}}(\kappa,1)$ such that 
		$T$ lies in the $\tfrac{1}{\lambda_{n'}}$‐neighborhood of $\tilde{T}$.
		
		For $\tilde{\phi}_{\lambda_{n'}}$, tubes pointing in different directions satisfy the polynomial Wolff axioms.
		
		\begin{claim}
			Let $S\subset{B}^n$ be a semi‐algebraic set of complexity at most $E$. Let 
			\[
			\tilde{\mathbb{T}}
			=\bigl\{\tilde{T}_{y_{\theta},v,\Phi_{\lambda_{n'}}}(\kappa,1)\bigr\}
			\]
			be a collection of tubes pointing in different directions. Then for any $\epsilon>0$, there is $C_{\epsilon,E}$ so that
			\[
			\#\{\tilde{T}\in\tilde{\mathbb{T}}:\tilde{T}\subset S\}
			\le C_{\epsilon,E}\,\mathcal{L}^n(S)\,\kappa^{-(n-1)-\epsilon}.
			\]
		\end{claim}
		\begin{proof}[Proof of Claim]
			Since $\tilde{\phi}_{\lambda_{n'}}$ is of the standard form \eqref{eq standard} without scaling, the polynomial Wolff axiom holds automatically.
		\end{proof}
		
		Since each $T\in T_{y_{\theta},v,\Phi_{\lambda_{n'}}}(\kappa,1)$ lies in a  
		$\tfrac{1}{\lambda_{n'}}$‐neighborhood of some $\tilde{T}\in\tilde{\mathbb{T}}(\kappa,1)$, we conclude
		\[
		\#\{T\in\mathbb{T}:T\subset S\}
		\le\#\{\tilde{T}\in\tilde{\mathbb{T}}:\tilde{T}\subset\mathcal{N}_{\tfrac{1}{\lambda_{n'}}}(S)\}
		\lesssim_{\epsilon,E}\mathcal{L}^{n}\bigl(\mathcal{N}_{\tfrac{1}{\lambda_{n'}}}(S)\bigr)\,\kappa^{-(n-1)-\epsilon}.
		\]
		This proves the second upper bound.
	\end{proof}
	Next, we bound the $n'$‐dimensional norm arising in the polynomial‐partitioning argument under the finite contact order assumption, yielding an improvement over the general positive‐definite case.  
	
	\begin{lemma}\label{lem improvement}
		Assume that the phase $\phi$ has contact order at most $l$. For every $\mathcal{S}_{n'}\in\mathcal{G}_{n'}$ and $r_{n'}=\mathcal{G}_{n'}$, we have
		\[
		\|f_{\nu,\mathcal{S}_{n'}}^{\star}\|_{2}^{2}
		\lesssim_{\epsilon}
		\min\bigl\{\lambda_{n'}^{l-(n-1)}\,r_{n'}^{-\frac{n-n'}{2}},\;(r_{n'}^{-\frac{1}{2}}+\tfrac{r_{n'}}{\lambda})^{n-n'}\bigr\}
		\,r_{n'}^{\sigma_{0}}\,
		\|f\|_{\infty}^{2},
		\]
		where $l$ is the contact order of $\phi$.
	\end{lemma}
	\begin{proof}
		We use Proposition \ref{prop scalingPWA} to prove this lemma. Since Proposition \ref{prop scalingPWA} describes properties of wave packets with unit length, we perform an appropriate rescaling first.
		
		Denote by $\tilde{\mathcal{S}}_{n'}$ the rescaled version of the semi‐algebraic set $\mathcal{S}_{n'}$ under scaling factor $r_{n'}$. This ensures that $\tilde{\mathcal{S}}_{n'}$ has unit scale in the tangent direction and scale $r_{n'}^{-\frac{1}{2}+\sigma_{n'}}$ in the normal direction. Then
		\[
		\mathcal{L}^{n}(\tilde{\mathcal{S}}_{n'})=r_{n'}^{-\frac{n-n'}{2}+\sigma_{n'}}.
		\]
		Let $\theta$ denote a frequency ball of radius $r_{n'}^{-\frac{1}{2}}$. Write $h=f_{\nu,\mathcal{S}_{n'}}^{\star}$ and perform wave packet decomposition
		\[
		h=\sum_{\theta}h_{\theta}.
		\]
		By local orthogonality,
		\begin{equation}\label{eq h}
			\|h\|_{2}^{2}
			\lesssim\sum_{\theta}\int_{\theta}\lvert h_{\theta}\rvert^{2}
			\lesssim\mathcal{L}^{n-1}(\theta)\,\#\bigl\{\theta:\,\text{there exists }v\text{ such that }T_{y_{\theta},v,\phi_{\lambda_{n'}}}\subset\tilde{\mathcal{S}}_{n'}\bigr\}\,\|f\|_{\infty}^{2}.
		\end{equation}
		Apply Proposition \ref{prop scalingPWA} with 
		\[
		\kappa=r_{n'}^{-\frac{1}{2}+\sigma},\quad
		S=\tilde{\mathcal{S}}_{n'}.
		\]
		Then
		\begin{equation}\label{eq theta}
			\begin{aligned}
				&\#\bigl\{\theta:\,\exists\,v\text{ s.t.\ }T_{y_{\theta},v,\phi_{\lambda_{n'}}}\subset\tilde{\mathcal{S}}_{n'}\bigr\}
				\;\le\;
				\#\{T\in\mathbb{T}:T\subset S\}\\
				&\le C_{\epsilon,E}\,
				\min\Bigl\{\mathcal{L}^{n}(\tilde{\mathcal{S}}_{n'})\,\lambda_{n'}^{l-(n-1)},\;
				\mathcal{L}^{n}\bigl(\mathcal{N}_{\,\tfrac{1}{\lambda_{n'}}}(\tilde{\mathcal{S}}_{n'})\bigr)\Bigr\}
				\,r_{n'}^{\frac{n-1}{2}+\sigma}\\
				&=
				C_{\epsilon,E}
				\min\Bigl\{\lambda_{n'}^{l-(n-1)}\,r_{n'}^{-\frac{n-n'}{2}+\sigma_{n'}},\;
				\bigl(r_{n'}^{-\frac{1}{2}+\sigma_{n'}}+\tfrac{1}{\lambda_{n'}}\bigr)^{n-n'}\Bigr\}\,
				r_{n'}^{\frac{n-1}{2}+\sigma}.
			\end{aligned}
		\end{equation}
		Combining \eqref{eq h} and \eqref{eq theta}, we obtain
		\[
		\|f_{\nu,\mathcal{S}_{n'}}^{\star}\|_{2}^{2}
		\lesssim_{\epsilon}
		\min\bigl\{\lambda_{n'}^{l-(n-1)}\,r_{n'}^{-\frac{n-n'}{2}},\;(r_{n'}^{-\frac{1}{2}}+\tfrac{r_{n'}}{\lambda})^{n-n'}\bigr\}
		\,r_{n'}^{\sigma_{0}}\,
		\|f\|_{\infty}^{2}.
		\]
	\end{proof}
	We are now ready to complete the proof via the polynomial method. In order to apply polynomial partitioning, we compute a coefficient in Lemma \ref{lem improvement} that is uniformly valid under our finite contact order hypothesis.

	If $r_{n'} \le \lambda^{\frac{2}{3}}$, then
	\[
	\|f_{\nu,\mathcal{S}_{n'}}^{\star}\|_{2}^{2}
	\lesssim_{\epsilon} r_{n'}^{-\frac{n-n'}{2}}\,
	r_{n'}^{\sigma_{0}}\,
	\|f\|_{\infty}^{2}.
	\]
	For $r_{n'} \ge \lambda^{\frac{2}{3}}$, the estimate
	\begin{equation}\label{eq improvement}
		\|f_{\nu,\mathcal{S}_{n'}}^{\star}\|_{2}^{2}
		\lesssim_{\epsilon} r_{n'}^{-\frac{(n-n')^{2}}{2(l-n+1)}}\,
		r_{n'}^{\sigma_{0}}\,
		\|f\|_{\infty}^{2}
	\end{equation}
	holds. Note that \eqref{eq improvement} in fact holds for all values of $r_{n'}$.
	
	For general H\"ormander positive definite phases, we recall the estimate for $f$ in the polynomial decomposition of dimension $n'$, where $m+1 \le n' \le n-1$. According to Proposition \ref{prop Property 3},
	\begin{equation}\label{eq general}
		\begin{aligned}
			\max_{\mathcal{S}_{n'}\in\mathcal{G}_{n'}}
			\|f_{\nu,\mathcal{S}_{n'}}^{\star}\|_{2}^{2}
			&\lesssim_{\epsilon}R^{O(\sigma)} 
			\Bigl(\frac{r_{n'+1}}{r_{n'}}\Bigr)^{-\frac{n-n'-1}{2}}
			D_{n'}^{-\,n'+\sigma}
			\max_{S_{n'+1}\in\mathcal{G}_{n'+1}}
			\|f_{\nu,S_{n'+1}}^{\star}\|_{2}^{2}\\
			&\lesssim_{\epsilon}R^{O(\sigma)}
			\prod_{i=n'}^{n-1}\Bigl(\frac{r_{i+1}}{r_{i}}\Bigr)^{-\frac{n-i-1}{2}}
			\prod_{i=n'}^{n-1}D_{i}^{-\,i+\sigma}\,
			\|f\|_{2}^{2}\\
			&\lesssim_{\epsilon}R^{O(\sigma)}
			r_{n'}^{\frac{n-n'}{2}}
			\prod_{i=n'}^{n-1}\bigl(r_{i}^{-\frac{1}{2}}\,D_{i}^{-\,i+\sigma}\bigr)\,
			\|f\|_{\infty}^{2}.
		\end{aligned}
	\end{equation}
	Set $\gamma\in[0,1]$ to be determined, and take the geometric average of \eqref{eq improvement} and \eqref{eq general}:
	\begin{equation}\label{eq average}
		\|f_{\nu,\mathcal{S}_{n'}}^{\star}\|_{2}^{2}
		\lesssim_{\epsilon}
		r_{n'}^{-\frac{(n-n')^{2}}{2(l-n'+1)}(1-\gamma)}
		\,r_{n'}^{\frac{n-n'}{2}\,\gamma}
		\prod_{i=n'}^{n-1}
		\bigl(r_{i}^{-\frac{\gamma}{2}}\,D_{i}^{-\,i\gamma+\gamma\sigma}\bigr)
		\,r_{n'}^{\sigma_{0}}
		\,\|f\|_{\infty}^{2}.
	\end{equation}
	Finally, applying the polynomial partitioning algorithm together with \eqref{eq average}, we complete the proof of Proposition \ref{prop Broad}. By Proposition \ref{prop Property 1},
	\begin{equation}\label{eq r1}
		\|\mathcal{T}^{\lambda}f\|_{BL_{k,A}^{p}(B_{R})}
		\lesssim_{\epsilon}R^{O(\sigma)}
		\prod_{i=m-1}^{n-1}
		\Bigl(r_{i}^{\frac{\beta_{i+1}-\beta_{i}}{2}}
		\,D_{i}^{\frac{\beta_{i+1}}{2}
			-\bigl(\frac{1}{2}-\frac{1}{p_{n}}\bigr)}\Bigr)\,
		\|f\|_{2}^{\frac{2}{p_{n}}}
		\max_{O\in\mathbf{n}_{l_{0}}^{\star}}
		\|f_{\nu,O}\|_{2}^{1-\frac{2}{p_{n}}}.
	\end{equation}
	By repeatedly applying Proposition \ref{prop Property 3}, we get
	\begin{equation}\label{eq r2}
		\max_{O\in\mathbf{n}_{l_{0}}^{\star}}
		\|f_{\nu,O}\|_{2}^{2}
		\lesssim_{\epsilon}R^{O(\sigma)}
		r_{n'}^{-\frac{n-n'}{2}}
		\prod_{i=m-1}^{n'-1}
		\bigl(r_{i}^{-\frac{1}{2}}\,D_{i}^{-\,i+\sigma}\bigr)
		\max_{\mathcal{S}_{n'}\in\mathcal{G}_{n'}}
		\|f_{\nu,\mathcal{S}_{n'}}^{\star}\|_{2}^{2}.
	\end{equation}
	Combining \eqref{eq r1} and \eqref{eq r2}, we have
	\begin{equation}\label{eq 1and2}
		\begin{aligned}
			\|\mathcal{T}^{\lambda}f\|_{BL_{k,A}^{p}(B_{R})}&
			\lesssim_{\epsilon}R^{O(\sigma)}
			\prod_{i=m-1}^{n-1}
			\Bigl(r_{i}^{\frac{\beta_{i+1}-\beta_{i}}{2}}
			\,D_{i}^{\frac{\beta_{i+1}}{2}
				-\bigl(\frac{1}{2}-\frac{1}{p_{n}}\bigr)}\Bigr)
			\,r_{n'}^{-\frac{n-n'}{2}\bigl(\frac{1}{2}-\frac{1}{p_{n}}\bigr)}\\
			&\quad\times
			\prod_{i=m-1}^{n'-1}
			\bigl(r_{i}^{-\frac{1}{2}
				\bigl(\frac{1}{2}-\frac{1}{p_{n}}\bigr)}\,D_{i}^{-\,i\bigl(\frac{1}{2}-\frac{1}{p_{n}}\bigr)+\sigma}\bigr)\,
			\|f\|_{2}^{\frac{2}{p_{n}}}
			\max_{\mathcal{S}_{n'}\in\mathcal{G}_{n'}}
			\|f_{\nu,\mathcal{S}_{n'}}^{\star}\|_{2}^{1-\frac{2}{p_{n}}}.
		\end{aligned}
	\end{equation}
	
	We apply \eqref{eq average} with fixed $n'$ in the polynomial partitioning algorithm. The gain arises at the $n'$‐dimensional step compared to standard polynomial methods. Taking any $n'$ already improves exponents over \cite{GHI}. Without loss of generality, set $n'=n-1$.
	
	To establish the broad estimate for $k=\frac{n+1}{2}$, consider the intersection of $p$‐ranges for all $m\ge k$. Since increasing $m$ reduces iteration count while widening the admissible $p$‐range, it suffices to take $m=\frac{n+1}{2}$.
	
	Define
	\[
	\begin{aligned}
		\tilde C_{1}
		&=\prod_{i=m-1}^{n-1}
		r_{i}^{\frac{\beta_{i+1}-\beta_{i}}{2}}
		\,D_{i}^{\frac{\beta_{i+1}}{2}
			-\bigl(\tfrac{1}{2}-\tfrac{1}{p_{n}}\bigr)},\\
		\tilde C_{2}
		&=r_{n-1}^{-\frac{1}{2}\bigl(\tfrac{1}{2}-\tfrac{1}{p_{n}}\bigr)}
		\prod_{i=m-1}^{n-2}
		r_{i}^{-\frac{1}{2}\bigl(\tfrac{1}{2}-\tfrac{1}{p_{n}}\bigr)}
		\,D_{i}^{-\,i\bigl(\tfrac{1}{2}-\tfrac{1}{p_{n}}\bigr)},\\
		\tilde C_{3}
		&=r_{n-1}^{-\frac{1-\gamma}{2(l-n+2)}
			\bigl(\tfrac{1}{2}-\tfrac{1}{p_{n}}\bigr)}
		\,D_{n-1}^{-\, (n-1)\,\gamma\bigl(\tfrac{1}{2}-\tfrac{1}{p_{n}}\bigr)}.
	\end{aligned}
	\]
	Using \eqref{eq 1and2} and \eqref{eq average}, we get
	\[
	\|\mathcal{T}^{\lambda}f\|_{BL_{k,A}^{p}(B_{R})}
	\lesssim_{\epsilon}
	\tilde C_{1}\,
	\tilde C_{2}\,
	\tilde C_{3}\,
	R^{C_{p}\,\sigma_{0}}\,
	\|f\|_{L^{2}}^{\frac{2}{p}}\,
	\|f\|_{L^{\infty}}^{1-\frac{2}{p}},
	\]
	where $C_{p}$ depends on $p$. Now we compute $\tilde{C} := \tilde{C}_{1} \tilde{C}_{2} \tilde{C}_{3}$. Since the term $R^{C_{p}\sigma_{0}}$ can be bounded by $R^{O(\epsilon)}$, we focus only on the product $\tilde{C}_{1} \tilde{C}_{2} \tilde{C}_{3}$.
	
	\begin{align*}
		\tilde C
		&=\tilde C_{1}\,\tilde C_{2}\,\tilde C_{3}
		\;\lesssim_{\epsilon}\;
		\prod_{i=m-1}^{n-2}
		r_{i}^{\frac{\beta_{i+1}-\beta_{i}}{2}
			-\frac{1}{2}\bigl(\tfrac{1}{2}-\tfrac{1}{p}\bigr)}
		\,D_{i}^{\frac{\beta_{i+1}}{2}
			-(i+1)\bigl(\tfrac{1}{2}-\tfrac{1}{p}\bigr)}
		\,r_{n-1}^{\frac{1-\beta_{n-1}}{2}}\\
		&\quad\times
		D_{n-1}^{\frac{1}{2}
			-\bigl(\tfrac{1}{2}-\tfrac{1}{p}\bigr)}
		\,r_{n-1}^{-\Bigl[\frac{1-\gamma}{2(l-n+2)}+\tfrac{1}{2}\Bigr]
			\bigl(\tfrac{1}{2}-\tfrac{1}{p}\bigr)}
		\,D_{n-1}^{-\, (n-1)\,\gamma\bigl(\tfrac{1}{2}-\tfrac{1}{p}\bigr)}.
	\end{align*}
	Choose $\beta_{m}\in[0,1]$ such that
	\[
	\frac{\beta_{m}}{2}
	- m\Bigl(\frac{1}{2}-\frac{1}{p}\Bigr)\ge 0.
	\]
	Since $r_{m-1}=1$, from \textbf{Output 5} with $i=m-1$ we get
	\begin{equation}\label{eq 5}
		D_{m-1}\le r_{m}.
	\end{equation}
	Using \eqref{eq 5} yields
	\begin{align*}
		\tilde C
		&\lesssim_{\epsilon} 
		r_{m}^{\frac{\beta_{m+1}}{2}
			-\bigl(\tfrac{1}{2}+m\bigr)\bigl(\tfrac{1}{2}-\tfrac{1}{p}\bigr)}
		\,D_{m}^{\frac{\beta_{m+1}}{2}
			-(m+1)\bigl(\tfrac{1}{2}-\tfrac{1}{p}\bigr)}
		\prod_{i=m+1}^{n-2}
		r_{i}^{\frac{\beta_{i+1}-\beta_{i}}{2}
			-\frac{1}{2}\bigl(\tfrac{1}{2}-\tfrac{1}{p}\bigr)}
		\,D_{i}^{\frac{\beta_{i+1}}{2}
			-(i+1)\bigl(\tfrac{1}{2}-\tfrac{1}{p}\bigr)}\\
		&\quad\times
		r_{n-1}^{\frac{1-\beta_{n-1}}{2}}
		\,D_{n-1}^{\frac{1}{2}
			-\bigl(\tfrac{1}{2}-\tfrac{1}{p}\bigr)}
		\,r_{n-1}^{-\Bigl[\frac{1-\gamma}{2(l-n+2)}+\tfrac{1}{2}\Bigr]
			\bigl(\tfrac{1}{2}-\tfrac{1}{p}\bigr)}
		\,D_{n-1}^{-\, (n-1)\,\gamma\bigl(\tfrac{1}{2}-\tfrac{1}{p}\bigr)}.
	\end{align*}
	Choose $\beta_{m+1}\in[0,1]$ so that
	\[
	\frac{\beta_{m+1}}{2}
	- \Bigl(m+\tfrac{1}{2}\Bigr)\Bigl(\tfrac{1}{2}-\tfrac{1}{p}\Bigr)\ge0.
	\]
	Since $D_{i}\ge1$ for $m\le i\le n$, \textbf{Output 5} gives
	\[
	r_{m}^{\frac{\beta_{m+1}}{2}
		-\bigl(\tfrac{1}{2}+m\bigr)\bigl(\tfrac{1}{2}-\tfrac{1}{p}\bigr)}
	\,D_{m}^{\frac{\beta_{m+1}}{2}
		-(m+1)\bigl(\tfrac{1}{2}-\tfrac{1}{p}\bigr)}
	\;\le\;
	r_{m+1}^{\frac{\beta_{m+1}}{2}
		-\bigl(\tfrac{1}{2}+m\bigr)\bigl(\tfrac{1}{2}-\tfrac{1}{p}\bigr)}.
	\]
	Hence
	\begin{align*}
		\tilde C
		&\lesssim_{\epsilon} 
		r_{m+1}^{\frac{\beta_{m+2}}{2}
			-(m+1)\bigl(\tfrac{1}{2}-\tfrac{1}{p}\bigr)}
		\,D_{m+1}^{\frac{\beta_{m+2}}{2}
			-(m+2)\bigl(\tfrac{1}{2}-\tfrac{1}{p}\bigr)}
		\prod_{i=m+2}^{n-2}
		r_{i}^{\frac{\beta_{i+1}-\beta_{i}}{2}
			-\frac{1}{2}\bigl(\tfrac{1}{2}-\tfrac{1}{p}\bigr)}
		\,D_{i}^{\frac{\beta_{i+1}}{2}
			-(i+1)\bigl(\tfrac{1}{2}-\tfrac{1}{p}\bigr)}\\
		&\quad\times
		r_{n-1}^{\frac{1-\beta_{n-1}}{2}}
		\,D_{n-1}^{\frac{1}{2}
			-\bigl(\tfrac{1}{2}-\tfrac{1}{p}\bigr)}
		\,r_{n-1}^{-\Bigl[\frac{1-\gamma}{2(l-n+2)}+\tfrac{1}{2}\Bigr]
			\bigl(\tfrac{1}{2}-\tfrac{1}{p}\bigr)}
		\,D_{n-1}^{-\, (n-1)\,\gamma\bigl(\tfrac{1}{2}-\tfrac{1}{p}\bigr)}.
	\end{align*}
	Impose 
	\[
	\frac{\beta_{i}}{2}
	-\Bigl(\frac{i-m}{2}+m\Bigr)\Bigl(\tfrac{1}{2}-\tfrac{1}{p}\Bigr)\ge0
	\quad
	\text{for }m\le i\le n.
	\]
	Then, by the same argument,
	\begin{align*}
		\tilde C
		&\lesssim_{\epsilon}
		r_{n-1}^{\frac{1}{2}
			- \Bigl[\tfrac{1-\gamma}{2(l-n+2)}
			+\tfrac{1}{2}+m+\tfrac{1}{2}(n-1-m)\Bigr]
			\bigl(\tfrac{1}{2}-\tfrac{1}{p}\bigr)}
		\,D_{n-1}^{\frac{1}{2}
			-\bigl(1+(n-1)\gamma\bigr)\bigl(\tfrac{1}{2}-\tfrac{1}{p}\bigr)}\\
		&=
		r_{n-1}^{\frac{1}{2}
			- \Bigl[\frac{1-\gamma}{2(l-n+2)}
			+\frac{n+m}{2}\Bigr]
			\bigl(\tfrac{1}{2}-\tfrac{1}{p}\bigr)}
		\,D_{n-1}^{\frac{1}{2}
			-\bigl(1+(n-1)\gamma\bigr)\bigl(\tfrac{1}{2}-\tfrac{1}{p}\bigr)}.
	\end{align*}
	Setting the exponents of $r_{n-1}$ and $D_{n-1}$ to be zero yields the critical equations
	\[
	\frac{1}{2}
	-\Bigl(\frac{3n+1}{4}+\frac{1-\gamma}{2(l-n+2)}\Bigr)
	\Bigl(\frac{1}{2}-\frac{1}{p}\Bigr)=0,
	\quad
	\frac{1}{2}
	=\bigl(1+(n-1)\gamma\bigr)
	\Bigl(\frac{1}{2}-\frac{1}{p}\Bigr).
	\]
	Solving for $\gamma$ and $p$ gives
	\[
	\gamma
	=\frac{(3n-3)(l-n+2)+2}
	{4(l-n+2)(n-1)+2},
	\]
	\[
	p
	=2\,\frac{3n+1}{3n-3}
	-\frac{4}{(3n-3)^{2}l+(3n-3)^{2}(2-n)+2(3n-3)}.
	\]
	This completes the proof of Proposition \ref{prop Broad}.

	\section{Curved maximal Kakeya estimate}\label{KKYM}
	\subsection{Definitions and reductions}
	The lower bounds for the Hausdorff dimension of curved Kakeya sets derived from oscillatory integrals often incur large losses. Consequently, we apply the polynomial method directly to curved maximal Kakeya estimates to obtain sharper Hausdorff dimension lower bounds and prove Theorem \ref{thm Kakeya}. 
	
	\begin{definition}[Curved Kakeya maximal function {\cite{Bourgain91.2}}]
		Given a H{\"o}rmander phase $\phi$ and $\varepsilon_\phi\in(0,\varepsilon_0]$, define
		\[
		\mathcal{K}_{\delta}^{\phi}f(y)
		\coloneqq 
		\sup_{\omega\in{B}_{\varepsilon_{\phi}}^{n-1}}
		\frac{1}{\delta^{n-1}}
		\int_{T_{y}^{\delta,\phi}(\omega)}\lvert f(x)\rvert\,dx,
		\]
		for all $y\in{B}_{\varepsilon_{\phi}}^{n-1}$ and $0<\delta<\varepsilon_{\phi}$.
	\end{definition}
	
	We shall work with the dual $L^p$ formulation for families of curved tubes.
	
	\begin{definition}[Direction-separated family]\label{dsf}
		Let $\mathbb{T}$ be a collection of $\delta$-tubes $T_{y}^{\delta}(\mathbf{x})$ associated with $(\phi,\varepsilon_\phi)$ (see Definition \ref{curved}).  For each $T\in\mathbb{T}$ write $y(T)=y$ for its direction. The family $\mathbb{T}$ is \emph{direction-separated} if 
		\[
		\lvert y(T_{1})-y(T_{2})\rvert > \delta
		\quad\text{for all }T_{1}\neq T_{2}\text{ in }\mathbb{T}.
		\]
	\end{definition}
	
	\begin{definition}[Curved maximal Kakeya estimate]
		Let $0<\delta<\varepsilon_{\phi}$ and let $\mathbb{T}$ be a direction-separated family of $\delta$-tubes in $\mathbb{R}^n$. We say the \emph{curved maximal Kakeya estimate} holds at exponent $p$ if for every $\epsilon>0$ there is $C_\epsilon>0$ such that
		\[
		\Bigl\|\sum_{T\in\mathbb{T}}\chi_{T}\Bigr\|_{L^{p}(\mathbb{R}^{n})}
		\le C_{\epsilon}\,\delta^{-\bigl(n-1-\tfrac{n}{p}\bigr)-\epsilon}
		\Bigl(\sum_{T\in\mathbb{T}}\lvert T\rvert\Bigr)^{\!1/p}
		\]
		for all such $\mathbb{T}$.
	\end{definition}

	Standard arguments translate curved maximal Kakeya estimates into lower bounds on the Hausdorff dimension of the associated curved Kakeya sets. If the curved maximal Kakeya estimate holds for all $p\ge p_0$, then the corresponding set $K$ satisfies
	\[
	\dim_{\mathcal{H}}K \ge p_0'.
	\]
	See \cite[Corollary 22.8, Theorem 22.9]{Mattila} and \cite[Proposition 3]{Wisewell05}.

	Next, we prove that for odd $n\ge3$, curved Kakeya sets $K$ associated with H{\"o}rmander phases of contact order $l$ satisfy
	\[
	\dim_{\mathcal{H}}K \ge \frac{n+1}{2} + \frac{n-1}{2(l+2-n)(n-1)+2}.
	\]
	In particular, in $\mathbb{R}^3$ the minimal contact order $A(3)=4$ yields
	\[
	\dim_{\mathcal{H}}K \ge 2+\tfrac{1}{7}.
	\]

	\begin{remark}\label{dnl}
		Similar to \cite{DGGZ24}, for positive-definite phases Theorem \ref{thm improvement} yields the  lower bound
		\[
		\dim_{\mathcal{H}}K \ge \frac{n+1}{2} + \frac{n-1}{4(l+2-n)(n-1)+2},
		\]
		for all odd $n\ge3$ and any contact order $l$. This matches the bound in \cite{DGGZ24}. In contrast, Theorem \ref{thm Kakeya} gives a stronger improvement without assuming positive-definiteness.
	\end{remark}
	
	To conclude Theorem \ref{thm Kakeya}, it suffices to prove the curved maximal Kakeya estimates.
	
	\begin{theorem}
		Let $n\ge3$ be odd, and suppose $\phi$ satisfies $(H_1)$ and $(H_2)$ and has contact order at most $l$ at the origin in $\mathbb{R}^n$ for some $l\geq A(n)$. { Then there exists a constant $\varepsilon_{\phi}>0$ satisfying the following property:} for any $\epsilon>0$ there exists $C_{\epsilon}>0$ such that
		\[
		\bigl\lVert\sum_{T\in\mathbb{T}}\chi_{T}\bigr\rVert_{L^{p}(\mathbb{R}^{n})}
		\le C_{\epsilon}\,\delta^{-\bigl(n-1-\tfrac{n}{p}\bigr)-\epsilon}
		\Bigl(\sum_{T\in\mathbb{T}}\lvert T\rvert\Bigr)^{1/p}
		\]
		holds whenever
		\[
		p\ge \frac{(l+2-n)(n-1)(n+1) + 2n}{(l+2-n)(n-1)^2 + 2(n-1)}.
		\]
		Here $\mathbb{T}$ is a direction-separated family of $\delta$-tubes associated with $(\phi,\varepsilon_\phi)$.
	\end{theorem}
	
	In the curved Kakeya problem, we use the broad-narrow method introduced by Bourgain and Guth \cite{BG11}. As in Section \ref{reduction}, we recall the broad norm.
	
	We decompose the direction set ${B}_{\varepsilon_{\phi}}^{n-1}$ into balls $\tau$ of radius $\rho$, where $\delta\ll\rho\ll1$. Let $\mathbb{T}_\tau$ be the collection of $\delta$-tubes with directions in $\tau$, and set $\mathbb{T}=\bigcup_\tau\mathbb{T}_\tau$. We also cover $\mathbb{R}^n$ by finitely overlapping balls $B_\delta$ of radius $\delta$, denoted by $\mathcal{B}_\delta$.
	
	Fix an integer $A$ to be chosen later. For each $B_\delta$ centered at $\mathbf{x}$, define
	\[
	\mu_{\mathbb{T}}(B_\delta)
	\coloneqq
	\min_{V_1,\dots,V_A\in Gr(k-1,n)}
	\max_{\substack{\tau\in\mathcal{U}(\mathbf{x},\rho^{-1},V_a)\\1\le a\le A}}
	\|\sum_{T\in\mathbb{T}_\tau}\chi_T\|_{L^p(B_\delta)}^p.
	\]
	Here $Gr(k-1,n)$ is the Grassmannian of $(k-1)$-planes in $\mathbb{R}^n$, and $k$ will be specified later. We define the $k$-broad norm by
	\begin{equation}\label{eq KKYnorm}
		\bigl\|\sum_{T\in\mathbb{T}}\chi_T\bigr\|_{BL^p_{k,A}(U)}
		\coloneqq
		\Bigl(\sum_{B_\delta\in\mathcal{B}_\delta}\frac{|B_\delta\cap U|}{|B_\delta|}
		\,\mu_{\mathbb{T}}(B_\delta)\Bigr)^{1/p},
	\end{equation}
	where $\mathcal{B}_\delta$ is any finitely overlapping cover of $\mathbb{R}^n$ by $\delta$-balls. For $p$ in the appropriate range, a $k$-broad estimate of this form implies the curved maximal Kakeya estimate.
	
	\begin{prop}[\cite{GHI,HRZ}]\label{prop KKYImply}
		Let $p\ge \tfrac{n-k+2}{n-k+1}$ and $A\approx1$. Let $\mathbb{T}$ be a direction-separated family of $\delta$-tubes. Suppose that for every $\epsilon>0$ there is $C_{\epsilon}>0$ such that
		\[
		\bigl\|\sum_{T\in\mathbb{T}}\chi_T\bigr\|_{BL^p_{k,A}(\mathbb{R}^n)}
		\le C_{\epsilon}\,\delta^{-(n-1-\tfrac{n}{p})-\epsilon}
		\Bigl(\sum_{T\in\mathbb{T}}|T|\Bigr)^{1/p}
		\quad\text{for all }0<\delta<1.
		\]
		Then for the same $p$ and every $0<\delta<1$,
		\[
		\bigl\|\sum_{T\in\mathbb{T}}\chi_T\bigr\|_{L^p(\mathbb{R}^n)}
		\le C_{\epsilon}\,\delta^{-(n-1-\tfrac{n}{p})-\epsilon}
		\Bigl(\sum_{T\in\mathbb{T}}|T|\Bigr)^{1/p}.
		\]
		
	\end{prop}
	
	\begin{remark}
		In \cite{HRZ}, Proposition \ref{prop KKYImply} is proved for the standard Kakeya estimate. The proof of Proposition \ref{prop KKYImply} requires only the non-degeneracy condition of H\"ormander phases, so we omit the proof here. Similar results are proved for the associated oscillatory integral operator in the variable-coefficient setting in \cite{GHI}.
	\end{remark}
	
	It then suffices to prove the following $k$-broad estimate.
	\begin{prop}\label{prop KKYBroad}
		Let $n\ge 3$ be odd, and let $k=\lceil (n+1)/2\rceil$. Suppose $\phi$ satisfies $(H_1)$ and $(H_2)$ and has contact order at most $l\in\mathbb N$ at the origin in $\mathbb R^n$. There exists a constant $\varepsilon_{\phi}>0$ with the following property: let $\mathbb T$ be a direction-separated family of $\delta$-tubes associated with $(\phi,\varepsilon_\phi)$. For any $\epsilon>0$ there is a constant $C_\epsilon>0$ such that
		\begin{equation}\label{eq KBroad}
			\bigl\|\sum_{T\in\mathbb T}\chi_T\bigr\|_{BL^p_{k,A}(\mathbb R^n)}
			\le C_\epsilon\,\delta^{-(n-1-\tfrac{n}{p})-\epsilon}
			\bigl(\sum_{T\in\mathbb T}|T|\bigr)^{1/p}
		\end{equation}
		whenever $p\ge \frac{(l+2-n)(n-1)k + n}{(l+2-n)(n-1)(k-1) + n-1}$.
	\end{prop}

	\begin{remark}    In fact, for any spatial dimension $n$ and contact order $l$, the $k$-broad estimate can be improved to 
		\[ p \geq \frac{k}{k-1} - \Theta(n,k,l) \quad \text{for any integer } 2\leq k \leq n, \]
		where $\Theta(n,k,l)$ is a small number. For even $n$, one shows
		\[
		p \ge \min_{2\le k\le n}\max\Bigl\{\frac{k}{k-1}-\Theta(n,k,l),\,\frac{n-k+2}{n-k+1}\Bigr\}
		= \frac{n+2}{n}.
		\]
		Indeed:
		\begin{itemize}
			\item $\frac{k}{k-1} = \frac{n-k+2}{n-k+1}$ exactly when $k = (n+2)/2$;
			\item for even $n$, $(n+2)/2$ is an integer;
			\item $\Theta(n,k,l)$ is very small.
		\end{itemize}
		Hence in even dimensions the curved maximal Kakeya estimates cannot be improved through this method. Indeed, for even $n$, any improvement of the Kakeya maximal estimate requires $l\approx n$, in line with the Bourgain's condition. However, since $A(n)\gg n$, our method cannot achieve this.
	\end{remark}
	
	\subsection{Proof of Proposition \ref{prop KKYBroad}}
	We continue to apply polynomial partitioning in the Kakeya setting, following a modification of the algorithm in \cite{HRZ}.
	
	Fix a sequence of exponents $p_{n'}$ satisfying
	\[
	p_k \ge p_{k+1} \ge \cdots \ge p_n =: p \ge 1,
	\quad k \le n'\le n,
	\]
	and define
	\[
	\iota_{n'} = \Bigl(1 - \tfrac1{p_{n'}}\Bigr)^{-1}\Bigl(1 - \tfrac1p\Bigr),
	\quad k \le n'\le n,
	\]
	so that $\iota_n=1$.  Fix $0<\epsilon'\ll\epsilon\ll1$.
	
	We now analyze the sum $\sum_{T\in\mathbb{T}}\chi_T$ at finer scales.
	
	\noindent{\bf Input.}
	\begin{itemize}
		\item A small scale $0<\delta\ll1$.
		\item A large integer $A\gg1$.
		\item A family $\mathbb{T}$ of $\delta$-tubes such that 
		\[
		\bigl\|\sum_{T\in\mathbb{T}}\chi_T\bigr\|_{BL^p_{k,A}}\neq 0.
		\]
	\end{itemize}
	
	\noindent{\bf Output.} At each $n'$‑dimensional step we produce:
	\begin{enumerate}[label=(\roman*)]
		\item A tuple of scales 
		\[
		\vec\delta_{n'}=(\delta_n,\delta_{n-1},\dots,\delta_{n'}),
		\quad
		\delta^{\epsilon'}=\delta_n>\cdots>\delta_{n'}\ge\delta^{1-\epsilon'}.
		\]
		\item A tuple of parameters 
		\[
		\vec D_{n'}=(D_n,D_{n-1},\dots,D_{n'}),
		\quad
		D_j\ge1\,(n'\le j\le n),\;D_n=1.
		\]
		\item Integers 
		\[
		A = A_n>A_{n-1}>\cdots>A_{n'}.
		\]
		\item A sequence of transverse complete intersections $\vec S_{n'}=(S_n,S_{n-1},\dots,S_{n'})$, which satisfies
		$$\dim S_i=i,\;\deg S_i=O(1).$$
		We denote the collection of $\vec S_{n'}$ by $\vec{\mathbb{S}}_{n'}$.
		\item  For each $\vec{S}_{n'}$, there is a ball $B[\vec{S}_{n'}]$ of radius $\delta_{n'}$ and a subfamily 
		$\mathbb{T}[\vec{S}_{n'}]\subset\mathbb{T}$ of tubes tangent to $S_{n'}$ in $B[\vec{S}_{n'}]$, given by
		\[
		\mathbb{T}[\vec{S}_{n'}] = \left\{ T \in \mathbb{T} : \left| T \cap \mathcal{N}_{\delta}(S_{n'}) \cap B[\vec{S}_{n'}] \right| \geq \delta_{n'}|T| \right\}.
		\]
	\end{enumerate}
	\begin{prop}\label{prop KKYproperty1}
		For each $n'\le n$, one has
		\[
		\begin{split}
			\Bigl\|\sum_{T\in\mathbb{T}}\chi_T\Bigr\|_{BL^p_{k,A}(\mathbb{R}^n)}
			&\;\lesssim\;
			\delta^{-\epsilon'}\,
			C(\vec D_{n'};\vec\delta_{n'})\,
			[\delta^n\#\mathbb{T}]^{1-\iota_{n'}}\\
			&\quad\times
			\Biggl(
			\sum_{\vec S_{n'}\in\vec{\mathbb{S}}_{n'}}
			\Bigl\|\sum_{T\in\mathbb{T}[\vec S_{n'}]}\chi_T\Bigr\|_{BL^{p_{n'}}_{k,A_{n'}}(B[\vec S_{n'}])}^{p_{n'}}
			\Biggr)^{\frac{\iota_{n'}}{p_{n'}}}\,.
		\end{split}
		\]
		Here
		\[
		C(\vec D_{n'};\vec\delta_{n'})
		=
		\prod_{i=n'}^{n-1}
		\Bigl(\frac{\delta_i}{\delta}\Bigr)^{\iota_{i+1}-\iota_i}
		D_i^{(1+\epsilon')(\iota_{i+1}-\iota_{n'})+\epsilon'}.
		\]
	\end{prop}

	\begin{prop}\label{prop KKYproperty2}
		For each $n'\le n-1$,
		\[
		\sum_{\vec S_{n'}\in\vec{\mathbb{S}}_{n'}}
		\#\mathbb{T}[\vec S_{n'}]
		\;\lesssim\;
		\delta^{-\epsilon'}\,
		D_{n'}^{1+\epsilon'}
		\sum_{\vec S_{n'+1}\in\vec{\mathbb{S}}_{n'+1}}
		\#\mathbb{T}[\vec S_{n'+1}].
		\]
	\end{prop}
	
	\begin{prop}\label{prop KKYproperty3}
		For each $n'\le n-1$,
		\[
		\max_{\vec S_{n'}\in\vec{\mathbb{S}}_{n'}}
		\#\mathbb{T}[\vec S_{n'}]
		\;\lesssim\;
		\delta^{-\epsilon'}\,
		D_{n'}^{-n'+\epsilon'}
		\max_{\vec S_{n'+1}\in\vec{\mathbb{S}}_{n'+1}}
		\#\mathbb{T}[\vec S_{n'+1}].
		\]
	\end{prop}
	
	The recursive process terminates when the stopping condition \cite[Section 7, The first algorithm]{HRZ} is met. We assume it stops at the $m$‑dimensional step, with $m\ge k$. For each $\vec S_m$, let $\mathcal{O}[\vec S_m]$ be the final collection of cells associated to $\vec S_m$ (see \cite[The first algorithm]{HRZ}), and set $\mathcal{O}=\bigcup_{\vec S_m\in\vec{\mathbb{S}}_m}\mathcal{O}[\vec S_m]$.
	
	Combining Propositions \ref{prop KKYproperty1} and \ref{prop KKYproperty2} with Properties I-III from \cite[Section 7, The first algorithm]{HRZ}, we obtain
	\begin{equation}\label{eq keyestimate}
		\begin{split}
			\bigl\|\sum_{T\in\mathbb{T}}\chi_T\bigr\|_{BL^p_{k,A}(\mathbb{R}^n)}
			&\lesssim
			\delta^{-O(\epsilon')}
			C(\vec D_m;\vec\delta_m)\,
			[\delta^n\#\mathbb{T}]^{1-\iota_m}\\
			&\quad\times
			\Biggl(
			\sum_{O\in\mathcal{O}}
			\bigl\|\sum_{T\in\mathbb{T}[O]}\chi_T\bigr\|_{BL^{p_m}_{k,A_{m-1}}(O)}^{p_m}
			\Biggr)^{\tfrac{\iota_m}{p_m}}.
		\end{split}
	\end{equation}
	
	Iterating Proposition \ref{prop KKYproperty2} as in \cite[Section 8]{HRZ} yields
	\begin{equation}\label{eq KKY1and2}
		\begin{split}
			\bigl\|\sum_{T\in\mathbb{T}}\chi_T\bigr\|_{BL^p_{k,A}(\mathbb{R}^n)}
			&\lesssim
			\delta^{-\epsilon'}
			\prod_{i=m-1}^{n-1}
			\Bigl(\frac{\delta_i}{\delta}\Bigr)^{\iota_{i+1}-\iota_i}
			D_{i+1}^{\iota_{i+1}-\bigl(1-\tfrac1p\bigr)+O(\epsilon')}\\
			&\quad\times
			\bigl(\max_{O\in\mathcal{O}}\#\mathbb{T}[O]\bigr)^{1-\tfrac1p}
			\,(\delta\sum_{T\in\mathbb{T}}|T|)^{\tfrac1p},
		\end{split}
	\end{equation}
	where $\delta_{m-1}=\delta$.
	
	Now we need to estimate $\max_{O \in \mathcal{O}} \#\mathbb{T}[O]$ under the finite contact order assumption $l$. Through repeated applications of Proposition \ref{prop KKYproperty3}, we obtain
	\begin{equation}\label{eq repeated}
		\begin{split}
			\max_{O\in \mathcal{O}}\#\mathbb{T}[O]
			&\;\lesssim\;
			\delta^{-\epsilon'}\,
			\Bigl(\prod_{i=m-1}^{n'-1}D_{i}^{-i+\epsilon'}\Bigr)\\
			&\quad\times
			\max_{\vec S_{n'}\in \vec{\mathbb{S}}_{n'}}\#\mathbb{T}[\vec S_{n'}],
		\end{split}
	\end{equation}
	for any $m\le n'\le n$.
	
	Without loss of generality, assume $S_{n'}\subset{B}_{r_{n'}}^n$, so that the semialgebraic set $\mathcal{N}_{\delta}(S_{n'})\cap B[\vec S_{n'}]$ is centered at the origin. 
	We use $\mathbb{T}_{S_{n'}}$ to denote the collection of tubes that satisfy the following properties:
	\begin{itemize}
		\item $|T_{1} \cap T_{2}| \leq \frac{1}{2}|T_1|$ for any $T_{1}, T_{2} \in \mathbb{T}_{S_{n'}}$;
		\item each tube has length $\delta_{n'}$ and width $\delta$;
		\item each tube is contained in $\mathcal{N}_{\delta}(S_{n'})\cap B[\vec S_{n'}]$.
	\end{itemize}
	Each of $T_{y, v}\in \mathbb{T}_{S_{n'}}$ has the form
	\[
	T_{y,v} = \{(x,t): |x - \Phi(v,t;y)| \le \delta,\; t \in [0,\delta_{n'}]\}.
	\]
	After rescaling the tubes to unit length, we have
	\begin{equation}\label{eq KKYscalingtube}
		T_{y,v}^{\circ}
		=\Bigl\{(x,t):\bigl|x-\tfrac{1}{\delta_{n'}}\Phi(\delta_{n'}v,\delta_{n'}t;y)\bigr|\le\tfrac{\delta}{\delta_{n'}},\ t\in[0,1]\Bigr\}.
	\end{equation}
	We rescale the semialgebraic set $S_{n'}$ so that its tangential length becomes $1$ and its normal length becomes $\tfrac{\delta}{\delta_{n'}}$, denoting the rescaled set by $S_{n'}^{\circ}$.
	Applying Proposition~\ref{prop scalingPWA} with $\lambda_{n'} = \tfrac{1}{\delta_{n'}}$, $S = S_{n'}^{\circ}$, and $\kappa = \tfrac{\delta}{\delta_{n'}}$ shows that for any $\epsilon' > 0$, there exists $C_{\epsilon',E}$ such that
	\begin{equation}\label{eq KKYS_{n'}}
		\#\mathbb{T}_{{S}_{n'}}
		\le C_{\epsilon',E}\,\min\Bigl\{\mathcal{L}^n(S_{n'}^\circ)\,\bigl(\tfrac1{\delta_{n'}}\bigr)^{l-(n-1)},\;\mathcal{L}^n\bigl(\mathcal{N}_{\delta_{n'}}(S_{n'}^\circ)\bigr)\}\,(\frac{\delta}{\delta_{n'}})^{-(n-1)-\epsilon'},
	\end{equation}
	Since every short tube in $\mathbb{T}_{S_{n'}}$ corresponds to $(\frac{1}{\delta_{n'}})^{n-1}$ tubes in $\mathbb{T}[\vec S_{n'}]$, we have 
	\begin{equation}\label{eq cor}
		\#\mathbb{T}[\vec S_{n'}]\lesssim (\frac{1}{\delta_{n'}})^{n-1}\#\mathbb{T}_{S_{n'}}.
	\end{equation}
	It is easy to compute that
	\begin{equation}\label{eq leb1}
		\mathcal{L}^n(S_{n'}^\circ)
		=\bigl(\tfrac{\delta}{\delta_{n'}}\bigr)^{n-n'},
	\end{equation}
	\begin{equation}\label{eq leb2}
		\mathcal{L}^n\bigl(\mathcal{N}_{\delta_{n'}}(S_{n'}^\circ)\bigr)
		=\bigl(\tfrac{\delta}{\delta_{n'}}+\delta_{n'}\bigr)^{n-n'}.
	\end{equation}
	Combining \eqref{eq KKYS_{n'}}-\eqref{eq leb2} yields
	\[
	\#\mathbb{T}[\vec S_{n'}]
	\le C_{\epsilon',E}\,
	\min\Bigl\{
	(\tfrac{\delta}{\delta_{n'}})^{n-n'}\,(\tfrac1{\delta_{n'}})^{l-(n-1)},\;
	(\tfrac{\delta}{\delta_{n'}}+\delta_{n'})^{n-n'}
	\Bigr\}
	\,\delta^{-(n-1)-\epsilon'}.
	\]
	For all $\delta_{n'}$, we have
	\begin{equation}\label{eq KKYuniform}
		\#\mathbb{T}[\vec S_{n'}]
		\le C_{\epsilon,E}\,
		(\tfrac{\delta}{\delta_{n'}})^{\frac{(n-n')^{2}}{l+1-n'}}
		\,\delta^{-(n-1)-\epsilon'}.
	\end{equation}
	
	For general H\"ormander phases $\phi$, we have
	\begin{equation}\label{eq KKYT1}
		\max_{O\in\mathcal{O}}\#\mathbb{T}[O]
		\lesssim \delta^{-\epsilon'}
		\Bigl(\prod_{i=m-1}^{n-1}D_i^{-i+\epsilon'}\Bigr)
		\,\delta^{-(n-1)}.
	\end{equation}
	If $\phi$ has finite contact order $l$, then
	\begin{equation}\label{eq KKYT2}
		\max_{O\in\mathcal{O}}\#\mathbb{T}[O]
		\lesssim \delta^{-\epsilon'}
		\Bigl(\prod_{i=m-1}^{n'-1}D_i^{-i+\epsilon'}\Bigr)
		\,\delta^{-(n-1)}
		\bigl(\tfrac{\delta}{\delta_{n'}}\bigr)^{\frac{(n-n')^{2}}{l+1-n'}}.
	\end{equation}
	Since \eqref{eq KKYT2} improves the bound for any $m\le n'\le n-1$, we henceforth set $n'=n-1$ for convenience.
	Let $\gamma\in[0,1]$ and take the geometric average of \eqref{eq KKYT1} and \eqref{eq KKYT2}:
	\begin{equation}\label{eq KKYaverage}
		\max_{O\in\mathcal{O}}\#\mathbb{T}[O]
		\;\lesssim\;
		\delta^{-\epsilon'}\,
		\Bigl(\prod_{i=m-1}^{n-2}D_{i}^{-i+\epsilon'}\Bigr)\,
		\bigl(\tfrac{\delta_{n'}}{\delta}\bigr)^{-\frac{\gamma}{l+2-n}}\,
		D_{n-1}^{-(n-1)(1-\gamma)}\,
		\delta^{-(n-1)}.
	\end{equation}
	
	Combining \eqref{eq KKY1and2} and \eqref{eq KKYaverage} gives
	\begin{equation}\label{eq KKYcombined}
		\begin{split}
			\bigl\|\sum_{T\in\mathbb{T}}\chi_T&\bigr\|_{BL^p_{k,A}(\mathbb{R}^n)}
			\;\lesssim\;
			\delta^{-O(\epsilon')}\,
			\bigl(\tfrac{\delta_{n-1}}{\delta}\bigr)^{1-\iota_{n-1}}
			\prod_{i=m-1}^{n-2}\bigl(\tfrac{\delta_i}{\delta}\bigr)^{\iota_{i+1}-\iota_i}
			\prod_{i=m-1}^{n-2}D_i^{\,\iota_{i+1}-(i+1)(1-\tfrac1p)}\\
			&\quad\times
			\bigl(\tfrac{\delta_{n-1}}{\delta}\bigr)^{-\frac{\gamma}{l+2-n}\,(1-\tfrac1p)}
			\,D_{n-1}^{\,1 - \bigl[(n-1)(1-\gamma)+1\bigr](1-\tfrac1p)}
			\,\delta^{-\bigl(n-1-\tfrac np\bigr)}
			\bigl(\sum_{T\in\mathbb{T}}|T|\bigr)^{1/p}.
		\end{split}
	\end{equation}
	Comparing with the form \eqref{eq KBroad}, it remains to control the coefficient
	\begin{equation}\label{eq KKYCoeff}
		\begin{split}
			&\bigl(\tfrac{\delta_{n-1}}{\delta}\bigr)^{\,1-\iota_{n-1}
				-\frac{\gamma}{l+2-n}\,(1-\tfrac1p)}\,D_{n-1}^{\,1 - \bigl[(n-1)(1-\gamma)+1\bigr](1-\tfrac1p)}\\
			&\quad\times
			\prod_{i=m}^{n-2}
			\bigl(\tfrac{\delta_i}{\delta}\bigr)^{\iota_{i+1}-\iota_i}
			\prod_{i=m-1}^{n-2}
			D_i^{\,\iota_{i+1}-(i+1)(1-\tfrac1p)}.
		\end{split}
	\end{equation}
	Recalling the definition of $\iota_i$ and the monotonicity $\iota_{i+1}\ge\iota_i$, we set
	\[
	\iota_m=\iota_{m+1}=\cdots=\iota_{n-1}
	=m\Bigl(1-\frac1p\Bigr).
	\]
	Consequently, for all $i$ with $m-1\le i\le n-2$,
	\[
	\iota_{i+1}-\iota_i\le0
	\quad\text{and}\quad
	\iota_{i+1}-(i+1)\Bigl(1-\frac1p\Bigr)\le0.
	\]
	To ensure non-positive exponents in the remaining terms, we impose
	\[
	1-\iota_{n-1}-\frac{\gamma}{\,l+2-n\,}\Bigl(1-\frac1p\Bigr)\le0
	\quad\text{and}\quad
	1-\bigl[(n-1)(1-\gamma)+1\bigr]\Bigl(1-\frac1p\Bigr)\le0.
	\]
	Solving these yields
	\[
	\gamma\ge\frac{(n-m)(l+2-n)}{(n-1)(l+2-n)+1}
	\quad\text{and}\quad
	p\ge\frac{(l+2-n)(n-1)m+n}{(l+2-n)(n-1)(m-1)+n-1}.
	\]
	Since the process terminates at the $m$‑dimensional step for all $m\ge k$, taking the intersection over $m\ge k$ gives
	\[
	p\ge\frac{(l+2-n)(n-1)k+n}{(l+2-n)(n-1)(k-1)+n-1},
	\]
	completing the proof of Proposition \ref{prop KKYBroad}.
	
	\begin{remark}
		For general H{\"o}rmander phases, only \eqref{eq KKYT1} holds, yielding the broad estimate for $p\ge \tfrac{k}{k-1}$. Hence, by Proposition \ref{prop KKYImply}, curved Kakeya sets have Hausdorff dimension at least $\lceil\tfrac{n+1}{2}\rceil$ for all $n\ge3$. In \cite{GHI}, the same sharp bound is obtained for positive-definite H{\"o}rmander phases, but here no positive-definiteness assumption is required.
	\end{remark}

	\section{Stability analysis and proof of Theorem \ref{thm generic}}\label{generic}
	
	In this section, we prove that the finite contact order condition is open and dense, thereby completing the proof of Theorem \ref{thm generic}. In particular, we show that the set of all functions with contact order $ \le l $ forms an open and dense subset in the space of H\"ormander phases.
	
	We denote the collection of all H\"ormander phases by $\mathbf H$ and the subset of phase functions with contact order $ \le l $ by $\mathbf C_l$. Since $\mathbf H^+$ is open in $\mathbf H$, it suffices to prove the following.
	
	\begin{prop}
		For any $l \ge A(n)$, the set $\mathbf C_l$ is open in $\mathbf H$. Moreover, $\mathbf C_l$ is dense in $(\mathbf H,\tau)$, where $\tau$ denotes the topology induced by the family of seminorms
		\[
		\phi \;\mapsto\; \|\partial^\alpha \phi\|_\infty,
		\]
		with multi-indices $\alpha \in \{0,1,2,\dots\}^{2n-1}$.
	\end{prop}
	
	\begin{proof}
		Fix the dimension $n$. To verify the contact order condition, we must check that the following matrix has full rank:
		\[
		\mathcal D =
		\begin{pmatrix}
			D_{i,j}^{(1)}(0)                & \cdots & D_{i,j}^{(l)}(0) \\[6pt]
			\vdots                          & \ddots & \vdots           \\[6pt]
			\bigl(D_{i_1,j_1}D_{i_2,j_2}\bigr)^{(1)}(0) & \cdots & \bigl(D_{i_1,j_1}D_{i_2,j_2}\bigr)^{(l)}(0) \\[6pt]
			\vdots                          & \ddots & \vdots           \\[6pt]
			\bigl(D_{i_1,j_1}D_{i_2,j_2}\cdots D_{i_k,j_k}\bigr)^{(1)}(0) 
			& \cdots & \bigl(D_{i_1,j_1}D_{i_2,j_2}\cdots D_{i_k,j_k}\bigr)^{(l)}(0) \\[6pt]
			\vdots                          & \ddots & \vdots           \\[6pt]
			\bigl(D_{i_1,j_1}D_{i_2,j_2}\cdots D_{i_{n-2},j_{n-2}}\bigr)^{(1)}(0)
			& \cdots & \bigl(D_{i_1,j_1}D_{i_2,j_2}\cdots D_{i_{n-2},j_{n-2}}\bigr)^{(l)}(0) \\[6pt]
			\vdots                          & \ddots & \vdots           \\[6pt]
			\textbf{D}^{(1)}(0)                      & \cdots & \textbf{D}^{(l)}(0)
		\end{pmatrix}.
		\]
		It is straightforward to see that if $\phi$ has contact order $\le l$ at the origin, then any phase function in a small enough neighborhood of $\phi$ also has contact order $\le l$. Hence the contact order condition is stable. Since the determinant depends continuously on the matrix entries, the set of phase functions with contact order at most $l$ is open in $C^\infty( B_{\varepsilon_0}^n)$. Moreover, $\mathbf H$ is open in $C^\infty( B_{\varepsilon_0}^n)$, so $\mathbf C_l$ is open in $\mathbf H$.
		
		For any $l_1 > l_2 \ge A(n)$, one has $\mathbf C_{l_2} \subset \mathbf C_{l_1}$. Thus it suffices to prove that $\mathbf C_{A(n)}$ is dense in $\mathbf H$ under the topology $\tau$. Without loss of generality, we may assume
		\[
		\phi(\mathbf x; y) \;=\; \langle x,y\rangle \;+\; t\,\langle y, A y\rangle \;+\; O\bigl(|\mathbf x|^2\,|y|^2 \;+\; |t|\,|y|^3\bigr).
		\]
		Since $D_{ij}(0)=0$ for all $1 \le i,j \le n-1$, when $l = A(n)$ the determinant $\det \mathcal D$ can be viewed as a polynomial $P$ in 
		\[
		\frac{(n-1)\,n}{2}\,\cdot A(n)
		\]
		variables consisting of the first $A(n)$ $t$-derivatives of each distinct Hessian entry. Its zero locus has dimension at most 
		\[
		\frac{(n-1)\,n}{2}\,\cdot A(n) \;-\; 1.
		\]
		For convenience, set
		\[
		N \;=\; \frac{(n-1)\,n}{2}\,\cdot A(n).
		\]
		For any $\phi$, the Hessian matrix $\nabla_y^2\phi$ determines a vector $Y\in \mathbb R^N$ whose components are the first $A(n)$ $t$-derivatives of each distinct Hessian entry. Given $\phi$, for any $\epsilon>0$ there exists $\overline Y \in  B_\epsilon^N(Y)$ with $P(\overline Y)\neq 0$, since the zero locus of $P$ has codimension at least one. Equivalently, for any $\phi\in \mathbf H$ and any $\epsilon>0$, there is $\bar\phi\in B_\epsilon(\phi)$ such that the contact order of $\bar\phi$ equals $A(n)$, where $B_\epsilon(\phi)$ denotes the $\epsilon$-open ball centered at $\phi$ in the $\tau$-topology.
	\end{proof}

	\section{Applications to Nikodym problems on Riemannian manifolds}\label{Application}
	
	In this section, we apply Theorem \ref{thm improvement} and Theorem \ref{thm Kakeya} to certain oscillatory integral operators and to the Nikodym problem on Riemannian manifolds.
	
	Fix a Riemannian manifold $\mathcal M$ and work locally. By rescaling the metric if necessary, we may assume we are inside a fixed unit ball $ B_{1}^n(0)\subset\mathbb R^n$ equipped with a Riemannian metric 
	$
	g =\{g_{ij}(\mathbf x)\}_{1\le i,j\le n}
	$
	that is quantitatively close to the flat metric. Given a constant $\varepsilon_{g}\in(0,1]$, fix two points $\textbf{x}_{0}=(0,\cdots,0)$ and $\textbf{y}_{0}=(0,\cdots ,0,\varepsilon_{g})$, where the Riemannian distance of two points less than the injectivity radius. Let 
	$a(\mathbf x,\mathbf y)$
	be a smooth function supported in $ B_{\varepsilon_g/10}^n(\textbf{x}_{0})\times B_{\varepsilon_g/10}^n(\textbf{y}_{0})$.
	
	The following operator lies at the core of many harmonic‐analysis problems on Riemannian manifolds—for instance, it is related to the spectral projection operator (see Lemma 5.1.3 in \cite{sogge2017fourier}), and to the Riemannian distance-set problem \cite{IosevichLiuXi2022}:
	\[
	T_\lambda^\mathcal M f(\mathbf x)
	\;=\;
	\int_{\mathcal M} e^{i\lambda\,d_g(\mathbf x,\mathbf y)}\,a(\mathbf x,\mathbf y)\,d\mathbf y,
	\]
	where $d_g(\mathbf x,\mathbf y)$ denotes the Riemannian distance.
	
	It is natural to ask: for which $p$ does
	\begin{equation}\label{eq BZ}
		\|\,T_\lambda^\mathcal M f\|_{L^p(\mathcal M)}
		\;\lessapprox\;
		\lambda^{-\frac n p}\,\|f\|_{L^p(\mathcal M)},
		\qquad
		\lambda\ge1,
	\end{equation}
	hold?  When $g$ is the flat metric, one expects \eqref{eq BZ} to hold for all $p\ge \tfrac{2n}{n-1}$, which is essentially equivalent to the Bochner--Riesz conjecture in Euclidean space.
	
	Minicozzi and Sogge \cite{MS} studied this operator on general Riemannian manifolds and showed that the range of exponents $p$ for which \eqref{eq BZ} can hold is strictly smaller than the conjectured Euclidean range. The boundedness of $T_\lambda^\mathcal M$ is closely connected to the Kakeya compression phenomenon, which in this setting refers to the compression of Nikodym sets.
	
	\begin{definition}[Nikodym sets on manifolds]\label{Def 8.1}
		Let $\Omega\subset B^n_1(0)$ be a Borel set and $0<\lambda<1$.  Define
		\begin{multline}
			\Omega_{\lambda}^{\star}
			\,=\,
			\bigl\{
			y \in  {B}_1^{n}(0)
			:\,
			\exists\, \gamma_{y} \,
			\text{with }\,
			\lvert \gamma_{y} \cap \Omega \rvert
			\,\geq\,\lambda
			\lvert \gamma_{y}\rvert,\\  \text{ where $\gamma_{y}$ is a unit geodesic segment passing through } y\text{} 
			\bigr\}.
		\end{multline}
		We say that $\Omega$ is a \textbf{Nikodym set} if, for all $\lambda$ sufficiently close to $1$, the set $\Omega_\lambda^\star$ has positive measure.
	\end{definition}
	
	On any Riemannian manifold $\mathcal M$, Nikodym sets must have Hausdorff dimension at least $\bigl\lceil\tfrac{n+1}2\bigr\rceil$, and this lower bound is sharp. Minicozzi and Sogge \cite{MS} showed that, even when $g$ is flat, one can perturb the metric so that Nikodym sets concentrate on a $\bigl\lceil\tfrac{n+1}2\bigr\rceil$-dimensional submanifold.  More precisely, given the Euclidean metric $g$ and any $\epsilon>0$, there is a Riemannian metric $\tilde g$ with 
	\[
	\|\tilde g - g\|_{C^k}\le\epsilon,
	\qquad
	k\ge1,
	\]
	such that Nikodym sets in $( B^n_1(0),\tilde g)$ lie on a $\bigl\lceil\tfrac{n+1}2\bigr\rceil$-dimensional totally geodesic submanifold.  Later, Sogge, the third author, and Xu \cite{Sogge-Xi} proved that if there exists any totally geodesic submanifold of dimension $\bigl\lceil\tfrac{n+1}2\bigr\rceil$, then—after a suitable perturbation of the metric—the Kakeya compression phenomenon occurs.
	
	A natural question arises: Under what curvature conditions can a Riemannian manifold overcome the Kakeya compression phenomenon? Sogge \cite{Sogge99} and Xi \cite{Xi17} proved that the Hausdorff dimension of Nikodym sets on constant curvature manifolds $M^n$ is at least $\tfrac{n+2}{2}$. Dai, Dong, Guo, and Zhang \cite{DGGZ24} showed that the Riemannian distance function satisfies Bourgain’s condition if and only if the manifold has constant curvature, and they improved the lower bound on the Hausdorff dimension of Nikodym sets in higher dimensions for such manifolds. Gao, the second author, and the third author proved that the Kakeya conjecture implies the Nikodym conjecture on Riemannian manifolds with constant curvature.
	
	Sogge \cite{Sogge99} further showed that on $3$‑dimensional manifolds with chaotic curvature, Nikodym sets have Minkowski dimension at least $\tfrac{7}{3}$ and Hausdorff dimension at least $\tfrac{9}{4}$. Dai, Dong, Guo, and Zhang then demonstrated that for three‑dimensional analytic manifolds with chaotic curvature, the Riemannian distance function has contact order $4$ at the origin. 
	
	In this paper, we consider a local variant of Nikodym sets on manifolds, associated with the Carleson--Sj\"olin operator $T_{\lambda}^{\mathcal M}$. Any such local Nikodym set is a Nikodym set in the sense of Definition \ref{Def 8.1}.\begin{definition}[Local Nikodym sets on manifolds]
		Given $\varepsilon_g\in(0,1)$, let $\Omega\subset B^n_{\varepsilon_{g}}(0)$ be a Borel set and $0<\lambda<1$. Define
		\begin{multline}
			\Omega_{\lambda}^{\star}
			\,=\,
			\bigl\{
			y \in  {B}_{{\varepsilon_{g}}/10}^{n}(\mathbf{y}_{0})
			:\,
			\exists\,x\in {B}_{\varepsilon_{g}/10}^{n}(\mathbf{x}_{0})\,\text{such that the  geodesic segment }\gamma_{x,y}  \\\text{ passing through }x\text{ and }y 
			\text{ satisfies }\,| \gamma_{x,y} \cap \Omega |
			\,\geq\,\lambda
			| \gamma_{x,y} |
			\bigr\}.
		\end{multline}
		We say that $\Omega$ is a \textbf{local Nikodym set} associated with $(g,\varepsilon_g)$ if, for all $\lambda$ sufficiently close to $1$, the set $\Omega_\lambda^\star$ has positive measure.
	\end{definition}
	Applying Theorem \ref{thm Kakeya}, we now prove that, for generic metrics on a $3$-dimensional manifold, every local Nikodym set has Hausdorff dimension at least $\frac{15}{7}$; we expect analogous results in higher dimensions.
	
	\begin{theorem}\label{thm Manifold}
		For any three-dimensional smooth manifold $M$, there is an open and dense set $\mathcal G$ of Riemannian metrics on $M$ such that, for every $g\in\mathcal G$, there exists $\varepsilon_{g}\in(0,1)$ such that every local Nikodym set associated with $(g,\varepsilon_g)$ on $(M,g)$ has Hausdorff dimension at least $\frac{15}{7}$.
	\end{theorem}
	
	\begin{proof}
		Local Nikodym sets on $(M,g)$ can be identified with curved Kakeya sets whose phase is the Riemannian distance, namely
		{
			\begin{equation}\label{eq distance}
				\phi(x,t;y)=d_g\bigl((x,t),(y,\varepsilon_{g})\bigr).    
			\end{equation}
		}
		One checks that $\phi$ satisfies $(H_1)$ and $(H_2^+)$. Hence by Theorem \ref{thm Kakeya}, if $d_g$ has contact order 4 at the origin then there exists a constant $\varepsilon_{g}$ such that local Nikodym sets have Hausdorff dimension at least $\tfrac{15}{7}$.
		
		For Riemannian metrics, contact order 4 at the origin is characterized by the nonvanishing condition
		\begin{equation}\label{eq metric}
			(g_{33,11}-g_{33,22})\,g_{33,123}
			- (g_{33,113}-g_{33,223})\,g_{33,12}
			\neq 0
		\end{equation}
		at the origin, given that $\varepsilon_g$ is sufficiently small; see \cite{DGGZ24}.
		
		The third jets of the metric at the origin form an algebraic variety in $\R^N$ for some integer $N$, constrained by polynomial relations stemming from curvature identities.
		Interpreting the left side of \eqref{eq metric} as a polynomial $\tilde P$ on this jet space, its zero locus inside the constraint variety has strictly smaller dimension, since $\tilde P$ does not vanish identically there. Therefore, the complement is open and dense in the space of Riemannian metrics. A $C^\infty$‑small perturbation of $g$ moves its jet into the set $\{\tilde P\neq 0\}$, ensuring contact order 4. Applying Theorem \ref{thm Kakeya} to $(M,\tilde g)$ completes the proof.
	\end{proof}

	{ To analyze $T_\lambda^\mathcal M$ for $a(\textbf{x};\textbf{y})$ supporting in ${B}_{\varepsilon_{g}/10}^{n}(\mathbf{x}_{0})\times{B}_{\varepsilon_{g}/10}^{n}(\textbf{y}_{0})$}, a standard reduction (cf.\ \cite{CS}, \cite{DGGZ24}) shows it suffices to study H\"ormander‐type oscillatory integrals with phase given by \eqref{eq distance}.  In particular, fix a hypersurface
	\[
	\mathcal M'
	\;=\;
	\{(\,y,\varepsilon_{g}\,): y\in B_{\varepsilon_{g}/10}^{n-1}(0)\}.
	\]
	Let $a(x,t;y)$ be a smooth function supported in $ B_{\varepsilon_{g}}^{n-1}(0)\times B_{\varepsilon_{g}}^{1}(0)\times B_{\varepsilon_{g}}^{n-1}(0)$ with 
	\[
	\mathrm{dist}\bigl(\mathrm{supp}_{(x,t)}\,a,\;\mathcal M'\bigr)\;>\;0.
	\]
	Then one studies
	\begin{equation}\label{eq Rcs}
		T_\lambda^\phi f(x,t)
		\;=\;
		\int_{\mathcal M'} e^{i\lambda\,\phi(x,t;y)}\,a(x,t)\,f(y)\,dy.
	\end{equation}
	The operator $T_\lambda^\phi$ can be regarded as a variable‐coefficient analog of the Fourier extension operator for the sphere. In \cite{gao2024refined}, Gao, Miao, and the third author employ multi‐scale wave‐packet techniques to establish sharp weighted $L^p$ estimates for $T_\lambda^\phi$, and consequently for $T_\lambda^\mathcal M$, in two dimensions.
	
	From Theorem \ref{thm GHI}, estimate \eqref{eq BZ} holds in the range
	\[
	p \;\ge\;
	\begin{cases}
		2\,\frac{3n+1}{3n-3}, & \text{if $n$ is odd},\\[6pt]
		2\,\frac{3n+2}{3n-2}, & \text{if $n$ is even}.
	\end{cases}
	\]
	Moreover, if $\phi$ has finite contact order at the origin $(0,0;0)$, then Theorem \ref{thm improvement} provides a better range of $p$ for the operator $T_\lambda^\phi$, which in turn yields improved $L^p$ bounds \eqref{eq BZ} for $T_{\lambda}^{\mathcal{M}}$.

	\section{Appendix}\label{Appendix}
	\begin{prop}\label{Prop}
		The matrix $\mathcal{D}$ consists of $A(n)$ rows, where 
		\begin{equation}\label{eq An}
			A(n) \;=\; 1 + \sum_{k=1}^{n-2} \Bigl[ 
			{\textstyle\binom{n-1}{k}}\,\tfrac{k! + I(k)}{2} 
			\;+\; 
			\sum_{i=0}^{k-1} 
			\tfrac{\binom{n-1}{i}\,\binom{n-1 - i}{\,k - i\,}\,\binom{n-1 - k}{\,k - i\,}}{2}
			\Bigl( k! \;-\; (\,k - i\,)!\,\tfrac{i! - I(i)}{2} \Bigr) 
			\Bigr].
		\end{equation}
		Here, $I(k)$ denotes the number of involutions in the symmetric group $S(k)$ of degree $k$, which satisfies the recurrence relation
		\[
		I(k) \;=\; I(k-1) + (k-1)\,I(k-2), 
		\quad I(0) = I(1) = 1.
		\]
	\end{prop}
	\begin{proof}
		
		To determine the number of rows in the matrix $\mathcal{D}$, it suffices to count the number of distinct terms generated by determinants of all $k$-th order minors of the symmetric matrix $D$, for $1 \le k \le n-2$. We begin by classifying the minors according to their order $k$. Denote the number of distinct terms generated by determinants of all $k$-th order minors as $B(n,k)$. Then
		\[
		A(n) \;=\; 1 + \sum_{k=1}^{n-2} B(n,k).
		\]
		
		Next, we further classify the $k$-th order minors based on the order of their symmetric sub-minors. The number of fully symmetric $k$-th order minors is 
		$\binom{n-1}{k}.$
		
		For any integer $0 \le i \le k-1$, define the set $\mathbf{M}(n,k,i)$ as:
		\begin{align*}
			\mathbf{M}(n, k, i) = \big\{ & \text{$k$-th order minors in the symmetric matrix $D$} \\
			& \text{that contain \textbf{exactly} an $i$-th order symmetric sub-minor} \big\}.
		\end{align*}
		Denote the cardinality of $\mathbf{M}(n,k,i)$ by $\#\mathbf{M}(n,k,i)$. One checks that
		\[
		\#\mathbf{M}(n,k,i)
		\;=\;
		\frac{\binom{n-1}{i}\,\binom{n-1 - i}{\,k - i\,}\,\binom{n-1 - k}{\,k - i\,}}{2}.
		\]
		
		Let $q(k)$ be the number of distinct terms in the determinant of a $k$-th order symmetric matrix. Also, let $p(k,i)$ be the number of distinct terms in the determinant of a $k$-th order minor that contains \emph{exactly} an $i$-th order symmetric sub-minor.
		
		\begin{lemma}
			The number of distinct terms in the determinant of a $k$-th order symmetric matrix satisfies
			\[
			q(k) \;=\; \frac{k! + I(k)}{2}.
			\]
			Here, $I(k)$ represents the number of \emph{involutions} in the symmetric group $S(k)$ of order $k$, i.e., the number of elements in $S(k)$ that are their own inverses.
		\end{lemma}
		\begin{proof}
			The determinant of a $k$-th order matrix has $k!$ terms (counting multiplicities), and each term corresponds to an element of $S(k)$.  If $\sigma \in S(k)$ is an involution (i.e.\ $\sigma = \sigma^{-1}$), it contributes one term; if $\sigma$ is not an involution, then $\sigma$ and $\sigma^{-1}$ contribute two terms. However, a couple of non-involutions contributes the same term in a symmetric matrix. Hence the involutions give $I(k)$ distinct terms, and the non-involutions give 
			\[
			\frac{k! - I(k)}{2}
			\]
			distinct terms.  Therefore,
			\[
			q(k) \;=\; I(k) + \frac{k! - I(k)}{2} \;=\; \frac{k! + I(k)}{2},
			\]
			with $I(k)$ satisfying
			\[
			I(k) \;=\; I(k-1) + (k-1)\,I(k-2), 
			\quad I(0) = I(1) = 1.
			\]
		\end{proof}
		
		\begin{lemma}
			The number of distinct terms $p(k,i)$ in the determinant of a $k$-th order minor that contains \emph{exactly} an $i$-th order symmetric submatrix is
			\[
			p(k,i) \;=\; k! \;-\; (\,k - i\,)!\,\frac{i! - I(i)}{2}.
			\]
			Here, $I(i)$ denotes the number of involutions in $S(i)$ of order $i$.
		\end{lemma}
		\begin{proof}
			The determinant of any $k$-th order matrix has $k!$ terms.  In its $i$-th order symmetric submatrix, there are 
			\[
			\frac{i! - I(i)}{2}
			\]
			duplicate terms.  Each duplicate extends in $(k - i)!$ ways to a term of the full $k$-th order determinant.  Thus,
			\[
			p(k,i) \;=\; k! \;-\; (\,k - i\,)!\,\frac{i! - I(i)}{2}.
			\]
		\end{proof}
		
		Based on this classification, substitute the counts into $B(n,k)$:
		\[
		B(n,k)
		\;=\;
		\binom{n-1}{k}\,q(k)
		\;+\;
		\sum_{i=0}^{k-1} \#\mathbf{M}(n,k,i)\,p(k,i).
		\]
		Hence
		\[
		A(n) 
		\;=\; 
		1 + \sum_{k=1}^{n-2} B(n,k),
		\]
		which yields the stated formula.
	\end{proof}

	\bibliography{reference}
	\bibliographystyle{alpha}

\end{document}